\documentclass[sn-mathphys,Numbered]{sn-jnl}

\usepackage{graphicx}%
\usepackage{multirow}%
\usepackage{amsmath,amssymb,amsfonts}%
\usepackage{amsthm}%
\usepackage{mathrsfs}%
\usepackage[title]{appendix}%
\usepackage{textcomp}%
\usepackage{manyfoot}%
\usepackage{booktabs}%
\usepackage{algorithm}%
\usepackage{algorithmicx}%
\usepackage{algpseudocode}%
\usepackage{listings}%
\usepackage{enumitem}
\usepackage{subcaption}
\usepackage{csvsimple}
\usepackage{datatool}
\usepackage{pgfplotstable}
\usepackage{amsmath}
\usepackage{bm}
\usepackage{bookmark}
\usepackage{thmtools}
\usepackage{listings}
\usepackage{multicol}
\usepackage{mathtools}

\lstdefinestyle{style1}{
    basicstyle=\ttfamily,
    keywordstyle =    \color{black}, 
    keywordstyle =    \color{black}, 
    commentstyle =    \color{black},
    stringstyle  =    \color{black},
    columns=fullflexible,
    keepspaces=true,
    upquote=true,
}

\newtheorem{theorem}{Theorem}
\newtheorem{proposition}[theorem]{Proposition}%
\newtheorem{remark}{Remark}%

\newcommand{\R}{\mathbb{R}}

\newcommand{\CS}{\mathbf{C_0}}
\newcommand{\cs}{\mathbf{c_0}}
\newcommand{\hadamard}{\odot}
\newcommand{\kron}{\otimes}
\DeclarePairedDelimiter\abs{\lvert}{\rvert}%
\DeclarePairedDelimiter\norm{\lVert}{\rVert}
\makeatletter
\let\oldabs\abs
\def\abs{\@ifstar{\oldabs}{\oldabs*}}
\let\oldnorm\norm
\def\norm{\@ifstar{\oldnorm}{\oldnorm*}}
\makeatother

\newcommand{\myvec}[1]{\textbf{\textup{vec}} ( #1)}
\newcommand{\diag}[1]{\textup{\textbf{diag}} ( #1)}
\newcommand{\ones}{\mathbf{1}}
\newcommand{\zeros}{\mathbf{0}}
\newcommand{\I}{\mathbf{I}}
\newcommand{\SN}{\bm{S}}
\newcommand{\TM}{\bm{T}}
\newcommand{\U}{\bm{U}}
\newcommand{\Gv}{\bm{G^v}}
\newcommand{\Gh}{\bm{G^h}}
\newcommand{\Xmatrix}{\bm{X}}
\newcommand{\Ymatrix}{\bm{Y}}
\newcommand{\Zmatrix}{\bm{Z}}
\newcommand{\Amatrix}{\bm{A}}
\newcommand{\Wv}{\bm{W^v}}
\newcommand{\Wh}{\bm{W^h}}
\newcommand{\Wmatrix}{\bm{W}}
\newcommand{\Cv}{\bm{C^v}}
\newcommand{\Ch}{\bm{C^h}}
\newcommand{\Vv}{\bm{V^v}}
\newcommand{\Vh}{\bm{V^h}}
\newcommand{\Vvk}[1]{\bm{V}^{\bm{v}, #1}}
\newcommand{\Vhk}[1]{\bm{V}^{\bm{h}, #1}}
\newcommand{\Wvk}[1]{\bm{W}^{\bm{v}, #1}}
\newcommand{\Whk}[1]{\bm{W}^{\bm{h}, #1}}
\renewcommand{\Bmatrix}{\bm{B}}
\newcommand{\Dv}{\bm{D^v}}
\newcommand{\Dh}{\bm{D^h}}
\newcommand{\Dmatrix}{\bm{D}}
\newcommand{\Cmatrix}{\bm{C}}
\newcommand{\Rmatrix}{\bm{R}}
\newcommand{\Pmatrix}{\bm{P}}
\newcommand{\Lambdamtrix}{\bm{\Lambda}}
\newcommand{\UTilde}{\bm{\tilde{U}}}
\newcommand{\VvkTilde}[1]{\bm{\tilde{V}}^{\bm{v}, #1}}
\newcommand{\VhkTilde}[1]{\bm{\tilde{V}}^{\bm{h}, #1}}

\pgfplotsset{compat=1.18}

\begin{document}

\title[Article Title]{Iteratively Reweighted Least Squares for Phase Unwrapping}

\author*[1]{\fnm{Benjamin} \sur{Dubois-Taine}}\email{benjamin.paul-dubois-taine@inria.fr}

\author[2,3]{\fnm{Roland} \sur{Akiki}\email{r.akiki@kayrros.com}}

\author[1,2]{\fnm{Alexandre} \sur{d'Aspremont}}\email{aspremon@ens.fr}


\affil*[1]{\orgdiv{DI ENS}, \orgname{Ecole Normale sup\'erieure, Universit\'e PSL, CNRS, INRIA}, \orgaddress{\city{Paris}, \postcode{75005}, \country{France}}}

\affil[2]{Kayrros SAS, Paris, 75009, France}

\affil[3]{Universit\'e Paris-Saclay, ENS Paris-Saclay, Centre Borelli, F-91190 Gif-sur-Yvette, France}


\abstract{The 2D phase unwrapping problem seeks to recover a phase image from its observation modulo $2\pi$, and is a crucial step in a variety of imaging applications. In particular, it is one of the most time-consuming steps in the interferometric synthetic aperture radar (InSAR) pipeline. In this work we tackle the $L^1$-norm phase unwrapping problem. In optimization terms, this is a simple sparsity-inducing problem, albeit in very large dimension. To solve this high-dimensional problem, we iteratively solve a series of numerically simpler weighted least squares problems, which are themselves solved using a preconditioned conjugate gradient method. Our algorithm guarantees a sublinear rate of convergence in function values, is simple to implement and can easily be ported to GPUs, where it significantly outperforms state of the art phase unwrapping methods.}



\maketitle

\section{Introduction}\label{sec:introduction}

We focus on the problem of estimating the phase of an input signal in scenarios where, because of the system's physical limitations, the phase is only available modulo $2\pi$. This is the case, for example, in interferometric synthetic aperture radar (InSAR) imaging~\cite{graham1974synthetic, zebker1986topographic, ghiglia1998two, rosen2000synthetic}, in
magnetic resonance imaging (MRI)~\cite{lauterbur1973image, hedley1992new} or in optical interferometry~\cite{pandit1994data}.

Formally, 2D phase unwrapping is the problem of recovering the true phase $\U$, given an observed wrapped phase $\Xmatrix$, such that \[\U = \Xmatrix + 2\pi \bm{K},\] where $\bm{K}$ is an integer-valued matrix. The problem is ill-posed, as any shift of $2\pi$ in the matrix $\U$ results in the same wrapped matrix $\Xmatrix$. Most phase unwrapping algorithms are based on the so-called Itoh condition~\cite{itoh1982analysis}, which states that the absolute difference between neighboring pixels is no more than $\pi$. When that condition is satisfied, unwrapping can be easily determined (up to a constant) by a simple integration procedure. However, in practice Itoh's condition is often violated. This can be due for example to the presence of noise in the input images, or to large discontinuities in the input signal. Phase unwrapping then becomes a much more complex task, which has been the focus of a long line of research~\cite{yu2019phase}. Most phase unwrapping algorithms fall within one of the following categories: path-following or optimization-based methods. We next give an overview of these two approaches. We refer the reader to the literature review in~\cite{yu2019phase} for a more complete description of modern phase unwrapping techniques.

\paragraph{\texorpdfstring{}{TEXT}Path following algorithms.} 

As previously mentioned, when the Itoh condition is satisfied, unwrapping can be exactly performed via a simple integration procedure, and the obtained result does not depend on the integration path. Integration can still be performed when the Itoh condition is not satisfied, but in that case the choice of the integration path significantly impacts the quality of the unwrapping. Path following methods are methods aimed at choosing good integration path. Among them are quality-guided methods, which choose the path based on a quality map of the input image, using this map to minimize phase unwrapping error in regions where the quality is high~\cite{flynn1996consistent, zhong2010improved, zhao2011quality, jian2016reliability}. Popular examples of quality maps include correlation maps, phase derivative variance maps, or priority maps. 

Another important class of path following algorithms are methods based on balancing residues in the image. Residues are the results of the loop integration of every $2 \times  2$ neighboring pixel block in the input image and are a key concept for many phase unwrapping techniques. One popular example of a residue-based method is the branch cut algorithm, which first connects nearby residues of opposite polarities, and then integrates without crossing the connections~\cite{goldstein1988satellite, ag2018interferometric}.

\paragraph{\texorpdfstring{$L^p$}{TEXT}-norm minimization methods.} $L^p$-norm minimization methods take a global approach to the phase unwrapping problem. Based on Itoh's condition, those methods seek to minimize the number of times the unwrapped phase differences fail to match the wrapped differences of the input image. In other words, the residual image formed from observed and reconstructed gradients should be sparse. This leads to minimization objectives of the form
\begin{align*}
    \min_{U} \sum_{i,j} |
\U_{i+1, j} - \U_{i,j} - \Gv_{i,j}|^p + \sum_{i,j} |\U_{i, j+1} - \U_{i,j} - \Gh_{i,j}|^p,
\end{align*}
where $\Gv_{ij}$ and $\Gh_{ij}$ denote are the vertical and horizontal phase differences, respectively, of the input image $\Xmatrix$. 

When $p=2$, the objective is quadratic, and efficient algorithms like preconditioned conjugate methods can be designed~\cite{ghiglia1994robust, ghiglia1998two}. Unfortunately, the output of the $L^2$-norm problem is much worse compared to the ones of $L^p$-norm problems with $p \leq 1$, as it tends to smooth out areas of large discontinuities~\cite{ghiglia1996minimum, ferretti2007insar}.

It is widely accepted that $p=0$ yields the best solution, since the problem is then exactly minimizing the number of pixels where the gradients do not match~\cite{yu2011residues}. Unfortunately, the objective function is not convex and the minimization problem is NP-hard~\cite{chen2000network}. Yet several methods have been proposed to approximately solve the non-convex problem~\cite{ghiglia1996minimum, ghiglia1998two, chen2000network, bioucas2007phase, yu2011residues}. In particular, \cite{ghiglia1996minimum} propose a general algorithm to solve the $L^p$-norm problem for any $p < 2$ by iteratively solving weighted linear systems. The proposed algorithm resembles an iteratively reweighted least squares approach when $p=1$, although the minimization is not exact with respect to the introduced weights, and thus no convergence proof is provided.

On the other hand, when $p=1$ the resulting problem is known to be a good approximation of the $L^0$-norm problem, and in that case the objective function is convex, so efficient algorithms can be developed~\cite{flynn1997two}. In particular, the $L^1$-norm problem can be cast as a minimum cost flow (MCF) which can be solved using graph solvers~\cite{costantini1998novel}.

Finally, statistics-based methods model the phase unwrapping problem as a maximum a posteriori (MAP) estimation problem. Different assumptions on the underlying probability distributions have been proposed, and those directly impact the hardness of the resulting optimization problem~\cite{nico2000bayesian, chen2001two, dias2002z, bioucas2007phase}. The SNAPHU method~\cite{chen2001two} is one such approach, and is one of the methods of choice for phase unwrapping, as evidenced by its use in many InSAR software packages~\cite{snap, hooper2012recent}. Unfortunately, statistics-based methods must solve complicated optimization problems in large dimension, and thus often exhibit long running times.

\paragraph{Contributions}
The main motivation behind our work is that despite this long line of research, the running time of existing methods can be prohibitive on very large images. For example, it is common in satellite imagery to treat images of a few thousand pixels by a few thousand pixels and as we will see in the experimental part of the paper, existing methods take at least a few minutes on such images. Our work significantly improves this running time, developing an algorithm with strong theoretical guarantees, which can be implemented on GPUs.

We focus on the $L^1$-norm problem. Our proposed method rewrites the $L^1$-norm as a weighted $L^2$-norm, introducing data-dependent weights~\cite{black1996unification, daubechies2004iterative, daubechies2010iteratively, bach2012optimization, mairal2014sparse, fornasier2016conjugate}. This leads to the so-called iteratively reweighted least squares (IRLS) algorithm, where the minimization is performed alternatively with respect to the weights and the objective image $\U$. The minimization with respect to the weights has a closed form solution, while the minimization with respect to $\U$ comes down to solving a linear system. This linear system is high-dimensional but the matrix-vector product associated is cheap to compute in practice, suggesting the use of iterative solvers. We thus implement the conjugate gradient (CG) method~\cite{nocedal2006conjugate} to approximately solve the system, and we provide a carefully designed preconditioner to  accelerate the CG iterations. We prove that despite the approximate solutions obtained through the CG algorithm, the iterates of the resulting IRLS algorithm satisfy a sublinear rate of convergence in function value.

We implement our fully GPU-compatible algorithm in Python. We show over extensive experiments on both simulated and real data that our method is competitive with state of the art phase unwrapping techniques in terms of image quality, and reduces the average running time by about one order of magnitude.

\paragraph{Paper organization}
The paper is organized as follows. We first introduce useful mathematical notations and properties. We detail the $L^1$-norm problem in Section~\ref{sec:model}. In Section~\ref{sec:alg} we describe the IRLS algorithm and detail the conjugate gradient method to solve the linear system that appears in IRLS iterations, as this is the main bottleneck of our method. We focus in particular on an efficient preconditioning of the linear system. We also prove sublinear convergence of the IRLS algorithm in terms of function value. In Section~\ref{sec:experiments} we compare our method against other phase unwrapping methods for different image sizes, both in terms of computing time and output quality.

\paragraph{Notation}
We write $\norm{\cdot}$ for the Euclidean norm if we are dealing with vectors, and for the Frobenius norm when dealing with matrices. We write the~$\ell_1$ norm of a vector or of a matrix as~$\norm{\cdot}_{\ell_1}$. 
We write $\I_n$ as the identity matrix of dimension $n$, $\ones_n$ and $\ones_{n\times m}$ as the vector and matrix whose entries are all equal to $1$. We define $\zeros_{n}$ and $\zeros_{n \times m}$ similarly. We drop the integer index when the dimension is clear from context.

The Hadamard (or component-wise) product of $\Xmatrix$ and $\Ymatrix$, written $\Xmatrix \hadamard \Ymatrix$, is the matrix of the same dimension as $\Xmatrix$ and $\Ymatrix$ with elements given by
\begin{align*}
    (\Xmatrix \hadamard \Ymatrix)_{i, j} = \Xmatrix_{i, j} \Ymatrix_{i, j}, \quad 1\leq i \leq n, \quad 1\leq j \leq m.
\end{align*}
We also write $\frac{1}{\Xmatrix}$ as the matrix whose entries are equal to the inverse of the entries in $\Xmatrix$, provided none of those are equal to zero.

The Kronecker product of two matrices $\Xmatrix\in\R^{n_1 \times m_1}$ and $\Ymatrix \in \R^{n_2 \times m_2}$, denoted as $\Xmatrix \kron \Ymatrix$, is the matrix of size $n_1 n_2 \times m_1 m_2$ given by
\begin{align*}
    \Xmatrix \kron \Ymatrix = \begin{pmatrix}
        \Xmatrix_{1, 1} \Ymatrix &\  \cdots &\  \Xmatrix_{1, m_1} \Ymatrix\\[5pt]
        \vdots &\  \ddots &\ \vdots \\[5pt]
        \Xmatrix_{n_1, 1} \Ymatrix &\ \cdots &\ \Xmatrix_{n_1, m_1} \Ymatrix
    \end{pmatrix}
\end{align*}
For a matrix $\Xmatrix \in \R^{n\times m}$, the vectorization of $\Xmatrix$ is the vector $\myvec{\Xmatrix}\in\R^{n m}$ defined as
\[
    \myvec{\Xmatrix} = (\Xmatrix_{1,1}, \Xmatrix_{2, 1}, \dots, \Xmatrix_{n, 1}, \Xmatrix_{1, 2}, \dots, \Xmatrix_{n, 2}, \dots, \Xmatrix_{1, m}, \dots \Xmatrix_{n, m})^\top.
\]
In particular, we will make use of the following property relating matrix-matrix products to matrix-vector products,
\begin{align}
    \label{prop:vec-kron}
    \myvec{\Xmatrix\Ymatrix\Zmatrix} = (\Zmatrix^\top \kron \Xmatrix) \myvec{\Ymatrix}.
\end{align}
For a vector $x \in \R^{n}$, we define $\diag{x}$ as the diagonal matrix of dimension $n \times n$ whose diagonal is equal to the vector $x$. For a matrix $\Xmatrix \in \R^{n \times n}$, we define $\diag{\Xmatrix}$ as the vector of size $n$ composed of the diagonal entries of $\Xmatrix$. We also define weighted norms. Formally, for a matrix $\Wmatrix \in \R^{n \times m}$ with strictly positive entries, the weighted Frobenius norm and weighted $\ell_1$ norm are respectively written
\[
    \norm{\Xmatrix}_{F, \Wmatrix} := \sqrt{\sum_{\substack{1 \leq i \leq n \\ 1 \leq j \leq m}} \Wmatrix_{i, j} (\Xmatrix_{i,j})^2},~~~ 
    \norm{\Xmatrix}_{\ell_1(\Wmatrix)} := ||\Wmatrix \hadamard \Xmatrix ||_{\ell_1} 
\]

\section{Model}\label{sec:model}
In this section we define the phase unwrapping problem, together with some useful notation. We assume that the input is a wrapped phase image $\Xmatrix \in \R^{N \times M}$, namely that $\Xmatrix_{i,j} \in [0, 2\pi)$ for all $1 \leq i\leq N$ and all $1\leq j \leq M$. The wrapped phase gradients are the matrices $\Gv \in \R^{(N-1) \times M}$ and $\Gh \in \R^{N \times (M-1)}$ defined as
\begin{align*}
    \Gv_{i,j} &= \Xmatrix_{i+1, j} - \Xmatrix_{i, j} \mod 2\pi \quad 1\leq i \leq N-1, \quad 1\leq j \leq M,\\
    \Gh_{i,j} &= \Xmatrix_{i, j+1} - \Xmatrix_{i, j} \mod 2\pi \quad 1\leq i \leq N, \quad 1 \leq j \leq M-1.
\end{align*}
Following \cite{costantini1998novel}, we solve the $L^1$-norm phase unwrapping problem defined as
\begin{align}
\label{prob:main}
    \min_{\U \in \R^{N \times M}} \left\{ \begin{aligned} &\sum_{\substack{1 \leq i \leq N-1 \\ 1 \leq j \leq M}} \Cv_{i,j} \left| \U_{i+1, j} - \U_{i, j} - \Gv_{i,j}\right| \\ +& \sum_{\substack{1 \leq i \leq N \\ 1 \leq j \leq M-1}} \Ch_{i,j} \left| \U_{i, j+1} - \U_{i, j} - \Gh_{i,j}\right| \end{aligned} \right\},
\end{align}
in the variable $\U \in \R^{N \times M}$, where $\Cv \in \R^{(N-1) \times M}, \Ch \in \R^{N \times (M-1)}$ are positive, user-defined weights. When problem~\eqref{prob:main} is solved using linear programming techniques, the resulting unwrapped image will be congruent to the input image, even if congruency is not explicity enforced. This is due to the total unimodularity of the constraint matrix introduced when casting~\eqref{prob:main} as a linear program~\cite{ahuja1988network, costantini1998novel}. However, because of our algorithm and the modifications we will make to problem~\eqref{prob:main}, there is no guarantee that our solution will be congruent to the input image. 
We argue that this is an acceptable approach for several reasons. First, most input images $\Xmatrix$ are noisy, and even when a denoising step is applied prior to the unwrapping process, unwrapped solutions might be more accurate when no congruency is enforced. Second, as we will see in the experimental section, our proposed algorithm yields close-to congruent outputs when the input image is noiseless. Finally, in terms of numerical efficiency, removing the congruency assumption allows for much faster algorithms in practice.

Problem~\eqref{prob:main} is ambiguous, in the sense that any solution yields the same objective function value under additive shift. We therefore define the following constraint set
\begin{align*}
    \CS = \left\{ \U \in \R^{N \times M} \mid \sum_{i,j} \U_{i, j} = 0 \right\}
\end{align*}
and consider the problem
\begin{align}
\label{prob:main+constraints}
    \min_{U \in\CS } \left\{ \sum_{\substack{1 \leq i \leq N-1 \\ 1 \leq j \leq M}} \Cv_{i,j} \left| \U_{i+1, j} - \U_{i, j} - \Gv_{i,j}\right| + \sum_{\substack{1 \leq i \leq N \\ 1 \leq j \leq M-1}} \Ch_{i,j} \left| \U_{i, j+1} - \U_{i, j} - \Gh_{i,j}\right|\right\},
\end{align}
We rewrite problem~\eqref{prob:main+constraints} in matrix form, defining matrices $\SN \in \R^{(N-1) \times N}$ and $\TM \in \R^{M \times (M-1)}$ as
\begin{align*}
    \SN_{i, j} &= \begin{cases}
        1 &\text{ if } j = i+1\\
        -1 &\text{ if } j = i\\
        0 & \text{ otherwise}
    \end{cases}, \quad 1 \leq i \leq N-1, \quad 1 \leq j \leq N, \\
    \TM_{i, j}  &= \begin{cases}
        1 &\text{ if } i = j+1\\
        -1 &\text{ if } i = j\\
        0 & \text{ otherwise}
    \end{cases}, \quad 1 \leq i \leq M, \quad 1 \leq j \leq M-1.
\end{align*}
Problem \eqref{prob:main+constraints} can then be cast as
\begin{align}
    \label{prob:main-matrix-form}
    \min_{\U\in \CS} \left\{  \norm{ \Cv \hadamard (\SN \U - \Gv)}_{\ell_1} + \norm{\Ch \hadamard (\U \TM - \Gh)}_{\ell_1} \right\}.
\end{align}

\section{Algorithm}\label{sec:alg}
This section details the algorithmic framework for solving~\eqref{prob:main-matrix-form}. 
We first rewrite~\eqref{prob:main-matrix-form} and introduce quadratic penalties. We cast the resulting problem as a sequence of weighted least squares problems whose solutions converge to the solution of the original problem. Each least squares problem is solved using the conjugate gradient method with an appropriate preconditioner. We detail each of these steps below, following the algorithmic framework in~\cite{fornasier2016conjugate}. Note that reweighting techniques similar in spirit (but not exactly the same) to the ones presented in this work have been proposed for the phase unwrapping problem~\cite{ghiglia1996minimum}, although no connection to the iteratively reweighted least squares approach was made and no convergence theory was provided.

\subsection{Iteratively reweighted least squares}\label{sec:IRLS}
We let $\Vv \in \R^{(N-1) \times M}$ and $ \Vh \in \R^{N \times (M-1)} $ be slack variables and rewrite problem~\eqref{prob:main-matrix-form} as
\begin{equation}
\begin{split}
    \label{prob:soft-constraints}
    \min_{\substack{\U\in \CS,\\ \Vv, \Vh}} F(\U, \Vv, \Vh) := \left\{ \begin{aligned}  & \norm{ \Vv}_{\ell_1( \Cv)} + \norm{\Vh}_{\ell_1(\Ch)} \\ &+ \frac{1}{2\tau} \norm{\SN \U - \Gv - \Vv}^2 + \frac{1}{2\tau} \norm{\U\TM - \Gh - \Vh}^2\end{aligned} \right\},
\end{split}
\end{equation}
for some penalization term $\tau > 0$. Problem~\eqref{prob:soft-constraints} is a quadratic penalty method for~\eqref{prob:main-matrix-form}, and iteratively solving~\eqref{prob:soft-constraints} for decreasing values of $\tau$ converging to 0 yields a sequence of solutions whose limit points are exact minimizers of~\eqref{prob:main-matrix-form}~\cite{bertsekas2014constrained}. We will however keep the value of $\tau$ constant to reduce the computing time. This is also motivated by the fact that images encountered in phase unwrapping problems are often noisy, and quadratic penalties might be beneficial in that regard. Moreover, substituting hard constraints for quadratic penalties has numerous numerical advantages, as pointed out in~\cite{fornasier2016conjugate}, and we shall explore them in Section~\ref{sec:CGM}.  Before detailing the iteratively reweighted least squares (IRLS) method, we prove existence and uniqueness of the solution of~\eqref{prob:soft-constraints}.

\begin{theorem}
\label{thm:solution-existence}
    Problem~\eqref{prob:soft-constraints} has a unique solution.
\end{theorem}

\begin{proof}
    Define the vectorized function $f$ as 
\begin{align*}
    f(u, v^v, v^h) := &\norm{v^v}_{\ell_1(c^v)} + \norm{v^h}_{\ell_1(c^h)} 
     + \frac{1}{2\tau} \norm{\bm{B} \begin{pmatrix}
        u \\ v^v \\ v^h
    \end{pmatrix}  - \begin{pmatrix}
        g^v \\ g^h
    \end{pmatrix}}^2.
\end{align*}
with $g^v = \myvec{\Gv}$, $g^h = \myvec{\Gh}$, and
\begin{align*}
    \bm{B} = \begin{pmatrix}
        \I_M \kron \SN &\ -\I_{(N-1)M} &\ \zeros \\
        \TM^\top \kron \I_N &\ \zeros &\ -\I_{N(M-1)}
    \end{pmatrix} 
\end{align*}
It follows that $F(\U, \Vv, \Vh) = f(\myvec{\U}, \myvec{\Vv}, \myvec{\Vh})$. Moreover define
\begin{align*}
    \cs = \{ x \in \R^{NM} \mid \sum_{i} x_i = 0\}.
\end{align*}
It is clear that $\U \in \CS \iff \myvec{\U} \in \cs$. Therefore problem~\eqref{prob:soft-constraints} is equivalent to
\begin{align*}
    \min_{u, v^v, v^h} f(u, v^v, v^h) + \delta_{\cs}(u)
\end{align*}
where $\delta_{\cs}$ is the indicator function of the closed convex set $\cs$. Now, the nullspace of the matrix $\Bmatrix$ is one dimensional and spanned by $\begin{pmatrix}
    \ones_{NM}^\top & \zeros_{(N-1)M}^\top & \zeros_{N(M-1)}^\top
\end{pmatrix}^\top$. 
    This implies that the feasible set is precisely the orthogonal complement of the nullspace of the matrix $\Bmatrix$. This in turn implies in particular that the quadratic function
\begin{align*}
    \frac{1}{2\tau} \norm{\Bmatrix\begin{pmatrix}
        u \\ v^v \\ v^h
    \end{pmatrix}  - \begin{pmatrix}
        g^v \\ g^h
    \end{pmatrix}}^2
\end{align*}
is strongly convex on the feasible set. Since $\norm{v^v}_{\ell_1(c^v)} + \norm{v^h}_{\ell_1(c^h)}$ is convex, the function $f$ is strongly convex on the feasible set, hence the minimizer is unique.
\end{proof}

Problem~\eqref{prob:soft-constraints} is a convex optimization problem, composed of a smooth part and non-smooth part. It could be tackled using well-established methods based on proximal steps such as FISTA~\cite{beck2009fast} or IHT~\cite{blumensath2009iterative}, but here we follow the IRLS approach \cite{lawson1961contribution, holland1977robust, osborne1985finite, daubechies2004iterative, daubechies2010iteratively, fornasier2016conjugate}. This is motivated by empirical evidence provided in~\cite{fornasier2016conjugate} showing the superiority of the IRLS approach in high dimension. We start by rewriting $F$ as
\begin{align}
    \begin{split}
    F(\U, \Vv, \Vh) = \quad &  \sum_{\substack{1 \leq i \leq N-1 \\ 1 \leq j \leq M}} \sqrt{ (\Cv_{i, j})^2 (\Vv_{i,j})^2}   +  \sum_{\substack{1 \leq i \leq N \\ 1 \leq j \leq M-1}} \sqrt{(\Ch_{i, j})^2 (\Vh_{i,j})^2}  \\
    + & \frac{1}{2\tau} \norm{\SN\U - \Gv - \Vv}^2 + \frac{1}{2\tau} \norm{\U\TM - \Gh - \Vh}^2.
    \end{split}
\end{align}
The issue in optimizing the above is the non-smooth square root part of the objective. For $x \in \mathbb{R}$, we have
\begin{align}
\label{eq:eta-trick}
    \abs{x} = \sqrt{x^2} = \frac{1}{2 } \inf_{w > 0} \left\{  \frac{x^2}{w} + w \right\}.
\end{align}
This simple equation is a classical way of rewriting sparsity-inducing norms in terms of quadratic objectives~\cite{black1996unification, daubechies2004iterative, daubechies2010iteratively, bach2012optimization, mairal2014sparse, fornasier2016conjugate}. Whenever $x \not = 0$, the above infimum is attained for $w = \abs{x}$. Unfortunately the infimum is not attained when $x = 0$. In our case this is an issue since our ultimate goal is to perform minimization with respect to this variable, so we introduce a fixed parameter $\delta > 0$ and observe that 
\begin{align}
    \sqrt{x^2 + \delta^2} = \frac{1}{2 } \inf_{w \geq \frac{\delta}{2}} \left\{ (x^2 + \delta^2)\frac{1}{w} + w \right\}.
\end{align}
Moreover, this infimum is always attained for $w =  \sqrt{x^2 + \delta^2}$. We therefore introduce the function $F_\delta$ defined as
\begin{align}
\label{eq:F_delta}
    \begin{split}
    F_\delta(\U, \Vv, \Vh) := \quad &  \sum_{\substack{1 \leq i \leq N-1 \\ 1 \leq j \leq M}} \sqrt{ (\Cv_{i, j})^2 (\Vv_{i,j})^2 + \delta^2}   +  \sum_{\substack{1 \leq i \leq N \\ 1 \leq j \leq M-1}} \sqrt{(\Ch_{i, j})^2 (\Vh_{i,j})^2 + \delta^2}  \\
    + & \frac{1}{2\tau} \norm{\SN\U - \Gv - \Vv}^2 + \frac{1}{2\tau} \norm{\U\TM - \Gh - \Vh}^2.
    \end{split}
\end{align}
We then show the following approximation result, as a function of $\delta$.

\begin{theorem}
\label{thm:F-vs-Fdelta}
    For any $\U \in \R^{N \times M}, \Vv\in \R^{(N-1) \times M}$ and $\Vh \in \R^{N \times (M-1)}$, 
    \begin{align*}
        F(\U, \Vv, \Vh) \leq F_{\delta}(\U, \Vv, \Vh) \leq F(\U, \Vv, \Vh) + \delta \left( (N-1)M + N (M-1)\right).
    \end{align*}
\end{theorem}
\begin{proof}
    It is clear that $F \leq F_\delta$ since $\delta > 0$. The second inequality comes from the fact that $\sqrt{a^2 + b^2} \leq a + b$ for $a, b \geq 0$ hence
    \begin{align*}
        F_\delta(\U, \Vv, \Vh) &\leq F(\U, \Vv, \Vh) + \sum_{\substack{1 \leq i \leq N-1 \\ 1 \leq j \leq M}} \delta  +  \sum_{\substack{1 \leq i \leq N \\ 1 \leq j \leq M-1}} \delta \\ &= F(\U, \Vv, \Vh) + \delta (N-1)M + \delta N(M-1)
    \end{align*}
    which is the desired result.
\end{proof}

Using the above remark, we get that minimizing $F_\delta$ is equivalent to solving 
\begin{equation}
\label{prob:alt1}
\begin{split}
    \inf_{\substack{\U\in \CS,\\ \Vv, \Vh,\\ \Wv \geq \delta/2, \Wh \geq \delta/2}}\left\{ \begin{aligned} 
    &\frac{1}{2} \sum_{\substack{1 \leq i \leq N-1 \\ 1 \leq j \leq M}}  ((\Cv_{i, j})^2 (\Vv_{i,j})^2 + \delta^2) \frac{1}{\Wv_{i,j}} + \Wv_{i,j} \\ + & \frac{1}{2} \sum_{\substack{1 \leq i \leq N \\ 1 \leq j \leq M-1}} ((\Ch_{i, j})^2 (\Vh_{i,j})^2 + \delta^2) \frac{1}{\Wh_{i,j}} + \Wh_{i,j} \\
    + & \frac{1}{2\tau} \norm{\SN\U - \Gv - \Vv}^2 + \frac{1}{2\tau} \norm{\U\TM - \Gh - \Vh}^2
    \end{aligned}\right\}
    \end{split},
\end{equation}
where $\Wv \in \R^{(N-1) \times M}$ and $\Wh \in \R^{M \times (N-1)}$.
We solve~\eqref{prob:alt1} by alternatively minimizing with respect to $(\U, \Vv, \Vh)$ and $(\Wv, \Wh)$. The objective is a simple convex function with respect to $\Wv$ and $\Wh$ and setting the gradients to $0$ yields the following closed-form updates for~$\Wv$ and~$\Wh$:
\begin{equation}
\begin{split}
\label{eq:weight-matrices-closed-form}
    \Wv_{i,j} &= \sqrt{(\Cv_{i, j})^2 (\Vv_{i,j})^2 + \delta^2}, \quad 1\leq i \leq N-1, \quad 1 \leq j \leq M, \\ 
    \Wh_{i,j} &= \sqrt{(\Ch_{i, j})^2 (\Vh_{i,j})^2 + \delta^2}, \quad 1\leq i \leq N, \quad 1 \leq j \leq M-1.
    \end{split}
\end{equation}
On the other hand, objective~\eqref{prob:alt1} is a quadratic function of $(\U, \Vv, \Vh)$, and the minimization can be performed using the following simple proposition.

\begin{proposition}
\label{prop:equivalence-min-linear-system}
    Minimizing~\eqref{prob:alt1} with respect to $(\U, \Vv, \Vh)\in \CS \times \R^{(N-1)\times M} \times \R^{N \times (M-1)}$ is equivalent to solving the system
    \begin{equation}
\label{eq:linear-system}
    \begin{pmatrix}
    \frac{1}{\tau} \left( \SN^\top \SN \U  + \U \TM\TM^\top - \SN^\top \Vv  - \Vh \TM^\top \right) \\[5pt]
        \Cv \hadamard \Cv \hadamard \frac{1}{\Wv} \hadamard \Vv + \frac{1}{\tau} (\Vv - \SN \U) \\[5pt]
        \Ch \hadamard \Ch \hadamard \frac{1}{\Wh} \hadamard \Vh + \frac{1}{\tau} (\Vh - \U \TM)
    \end{pmatrix} = \begin{pmatrix}
    \frac{1}{\tau} ( \SN^\top \Gv + \Gh \TM^\top)\\[5pt]
        -\frac{1}{\tau} \Gv\\[5pt]
        -\frac{1}{\tau} \Gh
    \end{pmatrix},
\end{equation}
for $(\U, \Vv, \Vh)\in \R^{N\times M} \times\R^{(N-1)\times M} \times \R^{N \times (M-1)}$, and then setting $\U = \U - \frac{\sum_{i, j} \U_{i,j}}{NM}$.
\end{proposition}
\begin{proof}
    \eqref{prob:alt1} is a quadratic function of $(\U, \Vv, \Vh)$, and setting its gradient with respect to those variables yields the linear system~\eqref{eq:linear-system}. Once we get a solution, and since the system is invariant by additive shift on $\U$, we can simply rescale it by its mean, i.e. set $\U$ to $\U - \frac{\sum_{i, j} \U_{i,j}}{NM}$ to ensure $\U\in \CS$.
\end{proof}

Solving~\eqref{eq:linear-system} exactly is too costly in practice. Instead, we suggest solving it approximately using the conjugate gradient method, which we detail in the next section. For now we summarize the above steps in Algorithm~\ref{algo-IRLS}.

\begin{algorithm}
\caption{IRLS algorithm}\label{algo-IRLS}
\begin{algorithmic}[1]
\Require $\Gv \in \R^{(N-1) \times M}, \Gh \in \R^{N \times (M-1)}, \Cv \in \R^{(N-1) \times M}, \Ch \in \R^{N \times (M-1)}, \delta >0.$
\State $\U^0 = \zeros_{N \times M}$
\State $\Vvk{0} = \SN \U - \Gv$, $\Vhk{0} = \U  \TM - \Gh$
\For{$k=0, 1, \dots$}
    \State $$\Wvk{k}_{i,j} = \sqrt{(\Cv_{i, j})^2 ((\Vvk{k})_{i,j})^2 + \delta^2}, \quad 1\leq i \leq N-1, \quad 1 \leq j \leq M $$
    \State $$\Whk{k}_{i,j} = \sqrt{(\Ch_{i, j})^2 ((\Vhk{k})_{i,j})^2 + \delta^2}, \quad 1\leq i \leq N, \quad 1 \leq j \leq M-1.$$
    \State For $(\U^{k+1}, \Vvk{k+1}, \Vhk{k+1})$, use Algorithm~\ref{algo-CGM} to approximately solve
    \begin{equation*}
    \begin{pmatrix}
     \SN^\top \SN \U^{k+1}  + \U^{k+1} \TM\TM^\top - \SN^\top \Vvk{k+1}  - \Vhk{k+1} \TM^\top  \\[5pt]
        \Cv \hadamard \Cv \hadamard \frac{1}{\Wvk{k}} \hadamard \Vvk{k+1} + \frac{1}{\tau} (\Vvk{k+1} - \SN \U^{k+1}) \\[5pt]
        \Ch \hadamard \Ch \hadamard \frac{1}{\Whk{k}} \hadamard \Vhk{k+1} + \frac{1}{\tau} (\Vhk{k+1} - \U^{k+1} \TM)
    \end{pmatrix} = \begin{pmatrix}
     \SN^\top \Gv + \Gh\TM^\top\\[5pt]
        -\frac{1}{\tau} \Gv\\[5pt]
        -\frac{1}{\tau} \Gh
    \end{pmatrix}.
\end{equation*}
\State $\U^{k+1} = \U^{k+1} - \frac{\sum_{ij}(\U^{k+1})_{ij}}{NM}$.
    \EndFor
\end{algorithmic}
\end{algorithm}

The next theorem gives a sublinear convergence rate for Algorithm~\ref{algo-IRLS}, provided the linear system~\eqref{eq:linear-system} is solved with enough accuracy. More precisely, it states that the algorithm converges as long as the function value at the iterates $(\U^{k+1}, \Vvk{k+1}, \Vhk{k+1})$ is smaller than the function value that would have been obtained if a simple gradient step had been taken on the previous iterates $(\U^k, \Vvk{k}, \Vhk{k})$.\\

\begin{theorem}
\label{thm:CV-rate}
Define $F_\delta(\U, \Vv, \Vh, \Wv, \Wh)$ as the function minimized in~\eqref{prob:alt1} and
\[
        C_{\text{max}} = \max ( \max_{ij} (\Cv_{ij}), \ \max_{ij} (\Ch_{ij})),\quad
        L = \frac{12}{\tau}  + \frac{1}{\delta} C_{\text{max}}^2. 
 \]
Let the candidate solutions $\UTilde^k, \VvkTilde{k}, \VhkTilde{k}$ at iteration $k$ be
    \begin{equation}
    \label{eq:candidate-gd-step}
        \begin{split}
            \UTilde^k &= \U^k - \frac{1}{\tau L}\left( \SN^\top \left( \SN \U^{k} - \Gv - \Vvk{k} \right)  + \left(\U^{k} \TM - \Gh - \Vhk{k} \right) \TM^\top \right) \\
            \VvkTilde{k} &= \Vvk{k} - \frac{1}{L} \left(  \Cv \hadamard \Cv \hadamard \frac{1}{\Wvk{k}} \hadamard \Vvk{k} + \frac{1}{\tau} (\Vvk{k} - \SN \U^{k} + \Gv) \right)\\
            \VhkTilde{k} &= \Vhk{k} - \frac{1}{L} \left(\Ch \hadamard \Ch \hadamard \frac{1}{\Whk{k}} \hadamard \Vhk{k} + \frac{1}{\tau} (\Vhk{k} - \U^{k} \TM + \Gh)\right)
        \end{split}
    \end{equation}
    Suppose that at each iteration of the IRLS algorithm~\ref{algo-IRLS}, the approximate solutions $(\U^{k+1}, \Vvk{k+1}, \Vhk{k+1})$ of the linear system~\eqref{eq:linear-system} satisfy
    \begin{align}
    \label{eq:sufficient-decrease}
        F_\delta(\U^{k+1}, \Vvk{k+1}, \Vhk{k+1}, \Wvk{k}, \Whk{k}) \leq F_\delta(\UTilde^{k}, \VvkTilde{k}, \VhkTilde{k}, \Wvk{k}, \Whk{k}) 
    \end{align}
    Then, after $K$ iterations, we have
    \begin{align}
    \label{eq:CV-rate}
       \boxed{F_\delta(\U^K, \Vvk{K}, \Vhk{K}) - F_\delta^* \leq \max \left( \left(\frac{1}{2}\right)^{\frac{K-1}{2}},  \frac{8L R^2}{K-1} \right),}
    \end{align}
    where $F_\delta^*$ is the optimal value of $F$ and the constant $R$ is defined as
    \begin{align}
        R &= \max_{\Ymatrix \in \Gamma} \left\{ \norm{\Ymatrix - \Ymatrix^*} \mid H_\delta(\Ymatrix) \leq H_\delta(\Ymatrix_0)\right\},
        \end{align}
    with $H_\delta$ being the function to be minimized in problem~\eqref{prob:alt1}, $\Gamma$ the feasible set of problem~\eqref{prob:alt1},
    \begin{align}
        \Ymatrix^* & = \begin{pmatrix}
            \U^* \\ \Vvk{*} \\ \Vhk{*} \\ \Wvk{*} \\ \Whk{*}
        \end{pmatrix} =  \begin{pmatrix}
            \U^* \\ \Vvk{*} \\ \Vhk{*} \\ \sqrt{\Cv \hadamard \Vvk{*} + \delta^2} \\ \sqrt{\Ch \hadamard \Vhk{*} + \delta^2} 
        \end{pmatrix}, \quad 
        \Ymatrix_0 = \begin{pmatrix}
            \U^0 \\ \Vvk{0} \\ \Vhk{0} \\ \Wvk{0} \\ \Whk{0}
        \end{pmatrix},
    \end{align} and $\begin{pmatrix}(\U^*)^\top, (\Vvk{*})^\top, (\Vhk{*})^\top\end{pmatrix}^\top$ 
        is the minimizer of $F_\delta$.
        \end{theorem}
\begin{proof}
    This is a direct consequence of~\cite[Theorem 4.1]{beck2015convergence}. Indeed, defining
        \begin{align*}
        s(u, v^v, v^h) = \frac{1}{2\tau} \norm{\begin{pmatrix}
        (\I_M \kron \SN) &\ -\I_{(N-1)M} &\ \zeros \\
        (\TM^\top \kron \I_N) &\ \zeros &\ -\I_{N(M-1)}
    \end{pmatrix} \begin{pmatrix}
        u \\ v^v \\ v^h
    \end{pmatrix}  - \begin{pmatrix}
        g^v \\ g^h
    \end{pmatrix}}^2,
    \end{align*}
    this result directly implies~\eqref{eq:CV-rate}, where $L$ is given by
    \begin{align}
    \label{eq:smoothness-def}
       L = L_s + \frac{1}{\delta} \lambda_{\text{max}} \left( \sum_{k=NM + 1}^{NM + (N-1) M} (c^{v}_{k})^2 e_{k} e_{k}^\top + \sum_{k=NM + (N-1)M + 1}^{MM + (N-1)M + N(M-1)} (c^h_{k})^2 e_{k } e_{k}^\top \right),
    \end{align}
    where $\{e_k\}_{k=1}^{NM + (N-1)M + N(M-1)}$ are the vectors in the canonical basis of $\R^{NM + (N-1)M + N(M-1)}$, $L_s$ is the Lipschitz constant of the gradient of $s$, which is equal to the maximum eigenvalue of the Hessian of $s$, and $\lambda_{\text{max}}(\cdot)$ denotes the maximum eigenvalue of the matrix in parentheses. We now bound $L_s$. The Hessian of $s$ is
    \begin{align*}
        \nabla^2 s(u, v^v, v^h) &= \frac{1}{\tau} \begin{pmatrix}
        (\I \kron \SN) &\ -\I &\ \zeros \\
        (\TM^\top \kron \I) &\ \zeros &\ -\I
    \end{pmatrix}^\top \begin{pmatrix}
        (\I \kron \SN) &\ -\I &\ \zeros \\
        (\TM^\top \kron \I) &\ \zeros &\ -\I
    \end{pmatrix}\\
    &= \frac{1}{\tau}\begin{pmatrix}
        (\I \kron \SN^\top \SN) + (\TM \TM^\top \kron \I) &\ -(\I \kron \SN)^\top &\ - (\TM \kron \I) \\
        -( \I \kron \SN) &\ \I &\ \zeros \\
        -(\TM^\top \kron \I) &\ \zeros &\ \I
    \end{pmatrix}
    \end{align*}
    We will make use of Gershgorin circle theorem to prove a bound on the largest eigenvalue of $\nabla^2 s$~\cite{gerschgorin1931uber}. The theorem states that all the eigenvalues of $\nabla^2 s$ lie within the union of the discs
    \begin{align}
    \label{eq:greshgorin-disk}
        D_i = \left\{ \lambda \in \R \mid |\lambda - (\nabla^2 s)_{ii}| \leq \sum_{j\not = i} |(\nabla^2 s)_{ij}| \right\},
    \end{align}
    for $1\leq i \leq NM + (N-1)M + N (M-1)$.
    
    Writing out the definitions of $\I_M \kron \SN$, $\I_M \kron \SN^\top \SN$, $\TM\TM^\top \kron \I_N$ and $\TM \kron \I_N$ (see Appendix~\ref{sec:matrix-definition}), one can see that all diagonal elements of $\nabla^2 s$ are positive (which can also be deduced from the fact that $\nabla^2 s$ is the Hessian of a convex quadratic) and bounded above by $\frac{4}{\tau}$, and that 
    \begin{align*}
        \max_{i }\sum_{j \not = i} |(\nabla^2s)_{ij}| = \frac{8}{\tau}.
    \end{align*}
    This value of $\frac{8}{\tau}$ is in particular attained for lines where the diagonal element is maximal, i.e. equal to $\frac{4}{\tau}$. This is the case for example at line $N+2$. Going back to the discs defined in~\eqref{eq:greshgorin-disk}, this implies that the maximum eigenvalue of $\nabla^2_s$, and therefore $L_s$, is such that
    \begin{align*}
        L_s - \frac{4}{\tau} \leq \frac{8}{\tau}, 
    \end{align*}
    and thus $ L_s \leq \frac{12}{\tau}$.

    We now bound the right part of~\eqref{eq:smoothness-def}, i.e. the maximum eigenvalue of
    \begin{align*}
         \left(  \sum_{k=NM + 1}^{NM + (N-1) M} (c^{v}_{k})^2 e_{k} e_{k}^\top + \sum_{k=NM + (N-1)M + 1}^{MM + (N-1)M + N(M-1)} (c^h_{k})^2 e_{k } e_{k}^\top \right).
    \end{align*}
    This is a diagonal matrix, whose first $NM$ elements are zero, and the remaining are the squares of the coefficient vectors $c^v$ and $c^h$. Therefore its maximum eigenvalue is given by
    \begin{align*}
  \max\left( \max_{k} \left\{(c^v_k)^2\right\}, \ \max_{k} \left\{(c^h_k)^2\right\} \right) = \max \left( \max_{ij} \left\{(\Cv_{ij})^2\right\}, \ \max_{ij} \left\{(\Ch_{ij})^2\right\}\right)
    \end{align*}
    hence the desired result.
\end{proof}

\begin{remark}
    The result in~\cite[Theorem 4.1]{beck2015convergence} is actually stated for an exact minimization step, i.e. when the linear system~\eqref{eq:linear-system} is solved exactly. However, upon closer inspection of the proof, one notices that the result holds as long as a sufficient decrease property holds. In our case this decrease property translates to condition~\eqref{eq:sufficient-decrease}.
\end{remark}

\begin{remark}
    In our experiments, we always checked that condition~\eqref{eq:sufficient-decrease} was satisfied after running CG. In practice this was always the case, even if CG was run for just a few iterations. We have also tried to implement Algorithm~\ref{algo-IRLS} using only the candidate updates~\eqref{eq:candidate-gd-step} instead of running CG, but this did not yield competitive performance.
\end{remark}

\subsection{Conjugate gradient}
\label{sec:CGM}

Since the updates for $\Wv$ and $\Wh$ have simple closed form solutions, the main numerical difficulty of Algorithm~\ref{algo-IRLS} relies in solving the linear system~\eqref{eq:linear-system}. In large dimensions such as the ones encountered in InSAR imagery, direct linear solvers will be too slow and we instead decide to use a carefully preconditioned conjugate gradient (CG) method~\cite{nocedal2006conjugate}. We detail the method in the specific case of problem~\eqref{eq:linear-system}. First let us simplify the notation by defining $n = NM + (N-1)M + N(M-1)$ and
\begin{align}\label{def:linear-system-simplification}
    \Bmatrix_1 := \frac{1}{\tau} (\SN^\top \Gv + \Gh \TM^\top),~\Bmatrix_2 := - \frac{1}{\tau} \Gv,~\Bmatrix_3 := -\frac{1}{\tau} \Gh.
\end{align}
 In order to make the algorithm more readable, we will vectorize the linear system using the following proposition. 

\begin{proposition}
    Defining the matrix $\Amatrix \in \R^{n \times n}$ as
    \begin{equation}
    \label{eq:matrix-soft-constraints}
        \Amatrix := \begin{pmatrix}
              \frac{1}{\tau} \left((\I_M \kron \SN^\top \SN) + (\TM\TM^\top \kron \I_N)\right) &\ \ -\frac{1}{\tau}(\I_M \kron \SN^\top) & \ \  -\frac{1}{\tau}(\TM \kron \I_N) \\[5pt]
            -\frac{1}{\tau}(
            \I_M \kron \SN) & \ \  \Dv + \frac{1}{\tau} \I &\ \ \zeros\\[5pt]
            -\frac{1}{\tau}(\TM^\top \kron \I_N) & \ \  \zeros & \ \    \Dh + \frac{1}{\tau} 
            \I
        \end{pmatrix},
    \end{equation}
    the linear system~\eqref{eq:linear-system} is equivalent to the following linear system
    \begin{align}
    \label{eq:linear-system-vectorized}
        \Amatrix \begin{pmatrix}
            u\\ v^v \\ v^h 
        \end{pmatrix} = b,
    \end{align}
    where 
    \begin{align}\label{def:linear-system-vectorized-simplification}
    \begin{split}
        \Dv &:=\diag{\myvec{\Cv \hadamard \Cv \hadamard \frac{1}{\Wv}}}, \\
        \Dh &:=\diag{\myvec{\Ch \hadamard \Ch \hadamard \frac{1}{\Wh}}}, \\
        b& := \begin{pmatrix}
            \myvec{\Bmatrix_1} \\ \myvec{\Bmatrix_2} \\ \myvec{\Bmatrix_3}
        \end{pmatrix}.
        \end{split}
    \end{align}
    Indeed,  $(\U, \Vv, \Vh)$ solves~\eqref{eq:linear-system} if and only if $(u, v^v, v^h)$ solves~\eqref{eq:linear-system-vectorized} with $u=\myvec{\U}$, $v^v = \myvec{\Vv}$, $v^h =\myvec{\Vh}$.
\end{proposition}
\begin{proof}
    This is a direct consequence of the vectorization property~\eqref{prop:vec-kron} and the fact that 
    \begin{align*}
        \myvec{\Cv \hadamard \Cv \hadamard \frac{1}{\Wv} \hadamard \Vv} &= \diag{\myvec{\Cv \hadamard \Cv \hadamard \frac{1}{\Wv}}} \myvec{\Vv} \\
        \myvec{\Ch \hadamard \Ch \hadamard \frac{1}{\Wh} \hadamard \Vh} &= \diag{\myvec{\Ch \hadamard \Ch \hadamard \frac{1}{\Wh}}} \myvec{\Vh}
    \end{align*}
\end{proof}


\begin{algorithm}
\caption{Conjugate gradient (CG)}\label{algo-CGM}
\begin{algorithmic}[1]
\Require $b \in \R^n$, matrix $A\in \R^{n \times n}$ or a function to compute the linear map $\Amatrix p$ for vectors $p$, preconditioner $\Dmatrix \succeq 0$.
\State Initialize $u^0, v^{v, 0}$ and $v^{h, 0}$.
\State Set $$r_l = b - \Amatrix \begin{pmatrix}
    u^0 \\ v^{v, 0} \\ v^{h, 0} 
\end{pmatrix}$$
\State Solve $\Dmatrix z_l = r_l$
\State $p_l = z_l$
\State $\rho_l = r_l^\top z_l $
\For{$l=0, 1, \dots$}
    \State $\alpha_l = \rho_l / (p_l^\top \Amatrix p_l)$
    \State Set \begin{align*}
        \begin{pmatrix}
    u^{l+1} \\ v^{v, l+1} \\ v^{h, l+1} 
\end{pmatrix} = \begin{pmatrix}
    u^l \\ v^{v, l} \\ v^{h, l}
\end{pmatrix} + \alpha_l p_l
    \end{align*}
    \State $r_{l+1} = r_l - \alpha_l \Amatrix p_l$
    \State Solve $\Dmatrix z_{l+1} = r_{l+1}$
    \State $\rho_{l+1} = (r_{l+1})^\top z_{l+1}$
    \State $\beta = \rho_{l+1}/ \rho_l$
    \State $p_{l+1} = z_{l+1} + \beta p_l$
    \EndFor
\end{algorithmic}
\end{algorithm}
For completeness, we detail the conjugate gradient method for solving problem~\eqref{eq:linear-system-vectorized} in Algorithm~\ref{algo-CGM}. We emphasize that in practice the system is never vectorized and the matrix $\Amatrix$ is never constructed. Instead, we keep all variables in matrix form and matrix vector products 
are computed using the linear map defined in~\eqref{eq:linear-system}. The reason we decided to introduce $\Amatrix$ in this section is two-fold. First it allows to write out the conjugate gradient method of Algorithm~\ref{algo-CGM} in what we believe to be a more reader-friendly way. Second, and more importantly, this way of writing will shed some light on how to pick the preconditioner, which is the key to fast convergence of the conjugate gradient method. This is the topic of the next section.

\begin{remark}
    \label{rem:ls-singular}
    The matrix $\Amatrix$ is singular. Its nullspace is one-dimensional, consisting of the vectors that are constant in the first $NM$ entries and $0$ everywhere else (this is due to the block $(\I \kron \SN^\top \SN) + (\TM \TM^\top \kron \I)$ being invariant under additive scaling). The conjugate gradient method still converges for singular matrices~\cite{kaasschieter1988preconditioned, hayami2018convergence}.
\end{remark}

\subsection{Choosing the preconditioner}\label{sec:precond}

The convergence of CG directly depends on the conditioning of the matrix~$\Amatrix$. In later iterations of IRLS when the weight matrices $\Wv$ and $\Wh$ have many small entries, the matrix $\Amatrix$ becomes very ill-conditioned. To overcome this issue, the most common trick is to use a preconditioner, i.e. a matrix $\Dmatrix = \Cmatrix^\top \Cmatrix$ such that $\Cmatrix^{-\top} \Amatrix \Cmatrix^{-1}$ is better-conditioned than $\Amatrix$, and such that linear systems of the form $\Dmatrix z = r$ are very cheap to solve in practice. For more on the topic of preconditioning, we refer the interested reader to \cite{kelley1995iterative, nocedal2006conjugate}. It is interesting to note that in~\cite{ghiglia1994robust}, the authors suggest using a preconditioned conjugate gradient method to solve the $L^2$-norm phase unwrapping problem, whose structure is close to the linear system we solve at each IRLS iteration. Their preconditioner is however different from the one we present next and proved to be less efficient in practice.

A common choice of preconditioner is the so-called Jacobi preconditioner, where $\Dmatrix$ is set to a diagonal matrix whose diagonal matches the diagonal of $\Amatrix$. In our case we choose the more efficient block diagonal preconditioner, i.e. we define $\Dmatrix$ as
\begin{align}
\label{eq:preconditioner-definition}
    \Dmatrix := \begin{pmatrix}
            \frac{1}{\tau} \left((\I \kron \SN^\top \SN) + (\TM\TM^\top \kron \I)\right)  &\ \ \zeros & \ \ \zeros \\[5pt]
            \zeros & \ \ \Dv + \frac{1}{\tau} \I &\ \ \zeros\\[5pt]
           \zeros & \ \  \zeros & \ \    \Dh + \frac{1}{\tau} \I
        \end{pmatrix}.
\end{align}
For solving $\Dmatrix z = r$, we proceed as follows. Writing 
\[
z = \begin{pmatrix}
    z_a^\top & z_b^\top & z_c^\top
\end{pmatrix}^\top \mbox{ and } r = \begin{pmatrix}
    r_a^\top & r_b^\top & r_c^\top
\end{pmatrix}^\top
\] 
the problem comes down to solving the following three smaller linear systems for $z_a, z_b$ and $z_c$ 
\begin{align}
    \label{eq:precond-subsystems}
    \left((\I \kron \SN^\top \SN) + (\TM\TM^\top \kron \I)\right) z_a &= \tau r_a \\
    (\Dv + \frac{1}{\tau} \I) z_b &= r_b \nonumber\\
    (\Dh + \frac{1}{\tau} \I) z_c &= r_c\nonumber
\end{align}
Since $\Dv$ and $\Dh$ are diagonal matrices, the last two subproblems are numerically fast to solve. Problem~\eqref{eq:precond-subsystems} can be recast as
\begin{align}
    \label{eq:lyapunov-equation}
    \SN^\top \SN \Zmatrix_a + \Zmatrix_a \TM\TM^\top = \tau \Rmatrix_a,
\end{align}
where $z_a = \myvec{\Zmatrix_a}$ and $r_a = \myvec{\Rmatrix_a}$. This equation is commonly known as a Sylvester equation (in the case where $N=M$, then $\TM = \SN^\top$, and the equation is known as a continuous-time Lyapunov equation). The Bartels-Stewart algorithm~\cite{bartels1972solution} is a common way to solve such equations by first computing the Schur decompositions of $\SN^\top \SN$ and $\TM\TM^\top$ and then solving for $\Zmatrix_c$ after some careful algebraic manipulations. In our case, since $\SN^\top \SN$ and $\TM\TM^\top$ are positive semidefinite matrices, the Schur decompositions are eigendecompositions
\begin{equation}
    \label{eq:eigendecompositions}
    \SN^\top \SN = \Pmatrix_S \Lambdamtrix_S \Pmatrix_S^\top, \quad \TM\TM^\top = \Pmatrix_T \Lambdamtrix_T \Pmatrix_T^\top,
\end{equation}
where $\Lambdamtrix_S, \Lambdamtrix_T$ are diagonal matrices, and $\Pmatrix_S$, $\Pmatrix_T$ are orthogonal matrices. Plugging this into~\eqref{eq:lyapunov-equation} gives
\begin{align}
    &\Pmatrix_S \Lambdamtrix_S \Pmatrix_S^\top \Zmatrix_a + \Zmatrix_a \Pmatrix_T \Lambdamtrix_T \Pmatrix_T^\top = \tau \Rmatrix_a\\
    \iff & \Lambdamtrix_S \Pmatrix_S^\top \Zmatrix_a \Pmatrix_T + \Pmatrix_S^\top \Zmatrix_a \Pmatrix_T \Lambdamtrix_T = \tau \Pmatrix_S^\top \Rmatrix_a \Pmatrix_T\\
    \iff & \left(( \I \kron \Lambdamtrix_S) + (\Lambdamtrix_T \kron \I) \right) \myvec{\Pmatrix_S^\top \Zmatrix_a \Pmatrix_T} = \tau \myvec{\Pmatrix_S^\top \Rmatrix_a \Pmatrix_T}\\
    \label{eq:lyapunov-diag-system}
    \iff & \left(( \I \kron \Lambdamtrix_S) + (\Lambdamtrix_T \kron \I) \right) \myvec{\Zmatrix'} = \tau \myvec{\Pmatrix_S^\top \Rmatrix_a \Pmatrix_T} \mbox{ and } \Zmatrix_a = \Pmatrix_S \Zmatrix' \Pmatrix_T^\top
\end{align}
Since $\Lambdamtrix_S$ and $\Lambdamtrix_T$ are diagonal, so is $( \I \kron \Lambdamtrix_S) + (\Lambdamtrix_T \kron \I)$, and system~\eqref{eq:lyapunov-diag-system} can be solved efficiently.

Observe that throughout the CG iterations, the linear system~\eqref{eq:lyapunov-equation} always has the same structure, and only the right-hand side $\Rmatrix_a$ changes. This implies that we can precompute the matrices $\Pmatrix_S, \Pmatrix_T, \Lambdamtrix_S, \Lambdamtrix_T$ before the start of the IRLS algorithm~\ref{algo-IRLS}  and keep them in memory. The main numerical bottleneck then lies in the computation of the matrix products $\Pmatrix_S^\top \Rmatrix_a \Pmatrix_T \text{ and } \Pmatrix_S \Zmatrix' \Pmatrix_T^\top$.

\subsubsection{Equivalence of preconditioned CG with better-conditioned linear system}\label{sec:eq-CG-PCG}

We point out that while using CG for singular systems is a well studied task (see~\cite{kaasschieter1988preconditioned, hayami2018convergence} and remark~\ref{rem:ls-singular}), in our case the preconditioner $\Dmatrix$ is also singular, but we claim that this is not an issue with the current choice of the preconditioner. The main reason is that both $\Amatrix$ and $\Dmatrix$ share the same one-dimensional nullspace, and that $r_k$ is always in the range of $\Amatrix$, and thus of $\Dmatrix$, so that the linear system on line 4 of Algorithm~\ref{algo-CGM} is always well-defined. 


 
More precisely, let us write the eigendecompositions of $\Amatrix$ and $\Dmatrix$ as
\[
        \Amatrix = \sum_{i=1}^n \lambda_i q_i q_i^\top, ~\Dmatrix = \sum_{i=1}^n \gamma_i d_i d_i^\top,
\]
where $\{q_i\}_{i=1}^n$ and $\{d_i\}_{i=1}^n$ are orthonormal bases of $\R^n$. Since $\Amatrix$ and $\Dmatrix$ have the same one-dimensional nullspace, we assume without loss of generality that $\lambda_1 = \gamma_1 = 0$ (namely that $\Amatrix q_1 = 0$ and $\Dmatrix d_1 = 0$), and that $q_1 = d_1$. We then get that $\lambda_i > 0$ and $\gamma_i > 0$ for all $i\geq 2$. Let us then define

\[
\Cmatrix = \sum_{i=2}^n \sqrt{\gamma_i} d_i d_i^\top,~\Cmatrix^* = \sum_{i=2}^n \frac{1}{\sqrt{\gamma_i}} d_i d_i^\top.
\]
    For simplicity, we will write $x_l = \begin{pmatrix}
        u^l \\ v^{v, l} \\ v^{h, l}
    \end{pmatrix} \in \R^n$.

    \begin{restatable}{proposition}{equivalencePrecond}
    \label{prop:equivalence-precond}
        Solving 
        \begin{align}
        \label{eq:linear-system-PCG}
            \Amatrix x = b
        \end{align}with conjugate gradient with preconditioner $\Dmatrix$ starting from $x_0$ is equivalent to solving the linear system      \begin{align}\label{eq:equivalent-linear-system-PCG}
        \Cmatrix^* \Amatrix \Cmatrix^* \tilde{x} = \Cmatrix^* b
        \end{align}
        with conjugate gradient without a preconditioner starting from $\tilde{x}_0 = \Cmatrix x_0$     
    \end{restatable}

    This is a known result when the matrix $\Dmatrix$ is non-singular, and our proof is basically the same. We provide it for completeness in Appendix~\ref{sec:proof-equivalence-precond}.

    

\subsubsection{Numerical benefits of the preconditioner}
To conclude this section, we focus on showing the numerical benefits of the preconditioner defined above. Before we do so, let us recall two important properties about the convergence of CG. Define
\begin{align*}
    \rho(\Amatrix) &:= \frac{\sqrt{\kappa(\Amatrix) } - 1}{\sqrt{\kappa(\Amatrix)} + 1}
\end{align*}
where $\kappa(\Amatrix)$ is the condition number of $\Amatrix$, namely the ratio of the largest eigenvalue over the smallest strictly positive eigenvalue. Then the residual $r_l = \Amatrix x_l - b$ converges to $0$ at the rate $\rho(\Amatrix)^l$ \cite{hayami2018convergence}. Now, for illustrative purposes, assume $\Amatrix$ is invertible. Then, for any polynomial $Q_l$ of degree $l$ such that $Q_l(0) = 1$, we have
\begin{align}
    \label{eq:CG-convergence-polynomial}
    \frac{\norm{
\Amatrix x_l - b}_{\Amatrix^{-1}}}{\norm{\Amatrix x_0 - b}_{\Amatrix^{-1}}} \leq \max_{z \in \sigma(\Amatrix)} | Q_l(z)|,
\end{align}
where $\sigma(\Amatrix)$ is the set of eigenvalues of $\Amatrix$~\cite[Corollary 2.2.1]{kelley1995iterative}. The same results hold for the convergence of the preconditioned system $\Cmatrix^* \Amatrix \Cmatrix^* \tilde{x} = \Cmatrix^* b$, in particular by defining $\rho(\Cmatrix^* \Amatrix \Cmatrix^*)$ similarly.

This convergence bound~\eqref{eq:CG-convergence-polynomial} highlights the fact that clustered eigenvalues can significantly help the convergence of CG. Indeed, the more clustered the eigenvalues are, the easier it is to find a polynomial $Q_l$ which has low absolute value for all of them. 

Although our matrices $\Amatrix$ and $\Cmatrix^* \Amatrix \Cmatrix^*$ are not invertible, we hypothesize that a result similar to~\eqref{eq:CG-convergence-polynomial} holds for singular matrices (just like it is true that the convergence rate in $\rho(\Amatrix)^l$ holds for singular matrices $\Amatrix$ \cite{hayami2018convergence}). We shift our focus to the comparison of the distribution of the (strictly positive) eigenvalues and of the condition number of $\Amatrix$ and~$\Cmatrix^* \Amatrix \Cmatrix^*$. To do so, we generate matrices $\Dv$ and $\Dh$ with entries uniformly distributed between $0$ and $1/\delta$ (since, assuming $\Cv$ and  $\Ch$ are unity, this is the interval in which the entries of $\Dv$ and $\Dh$ lie). With those, we construct $\Amatrix$ and $\Dmatrix$ from~\eqref{eq:matrix-soft-constraints} and~\eqref{eq:preconditioner-definition} respectively. We set $\delta = 10^{-6}$ as in the experimental section~\ref{sec:experiments}.

We plot in Figure~\ref{fig:eigenvalues-distribution} the spectrum of the matrix $\Amatrix$ and the spectrum of the matrix $\Cmatrix^* \Amatrix \Cmatrix^*$ for $N = M = 16$. We observe a nice clustering of the eigenvalues of $\Cmatrix^*\Amatrix\Cmatrix^*$ around 1. On the other hand, the distribution of the eigenvalues of $\Amatrix$ seem to be almost continuous from about 1 to over $10^5$.

\begin{figure}[ht]

\centering
\begin{subfigure}{.5\textwidth}
  \centering
  \includegraphics[width=.9\linewidth]{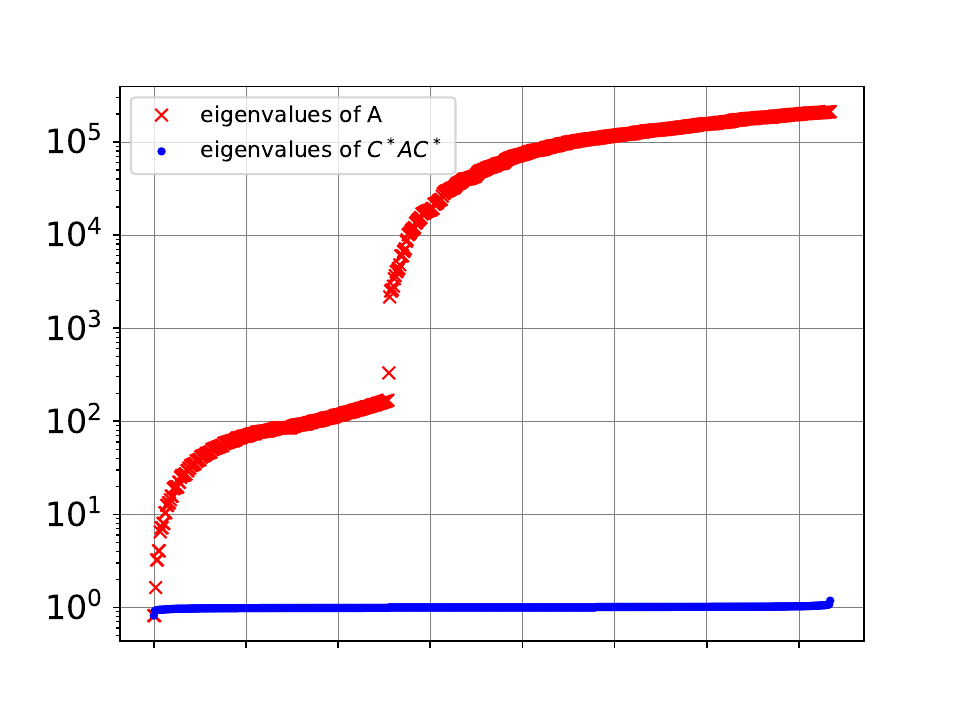}
\end{subfigure}%
\begin{subfigure}{.5\textwidth}
  \centering
  \includegraphics[width=.9\linewidth]{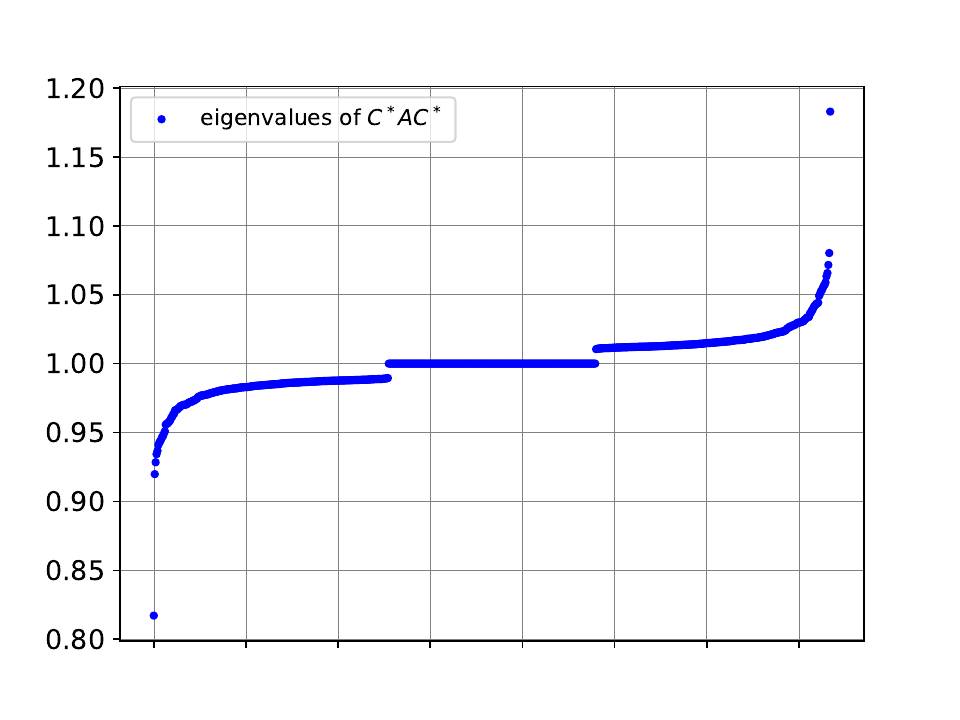}
\end{subfigure}
\caption{Left: Distribution of the eigenvalues of matrices $\Amatrix$ and $\Cmatrix^* \Amatrix \Cmatrix^*$ for $N = M = 16$ in log scale. Right: distribution of the eigenvalues of matrix $\Cmatrix^* \Amatrix \Cmatrix^*$ only. We observe that the eigenvalues of the preconditioned matrix are significantly more clustered.}
\label{fig:eigenvalues-distribution}
\end{figure}

Figure~\ref{fig:condition-numbers} plots both the condition number and the values of $\rho$ for $\Amatrix$ and $\Cmatrix^* A \Cmatrix$, for different values of $N$. We observe that as the dimension increases, the condition number of $\Amatrix$ explodes and the value of $\rho$ is very close to 1. For the matrix $\Cmatrix^* \Amatrix \Cmatrix^*$, the conditioning is much better.

\begin{figure}[ht]

\centering
\begin{subfigure}{.5\textwidth}
  \centering
  \includegraphics[width=.9\linewidth]{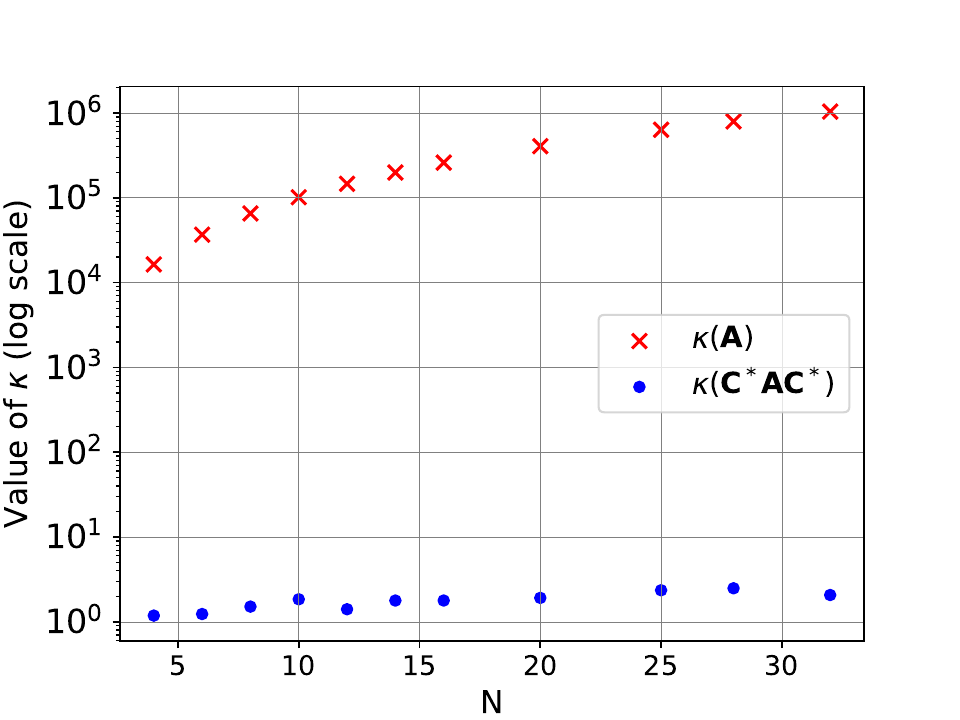}
\end{subfigure}%
\begin{subfigure}{.5\textwidth}
  \centering
  \includegraphics[width=.9\linewidth]{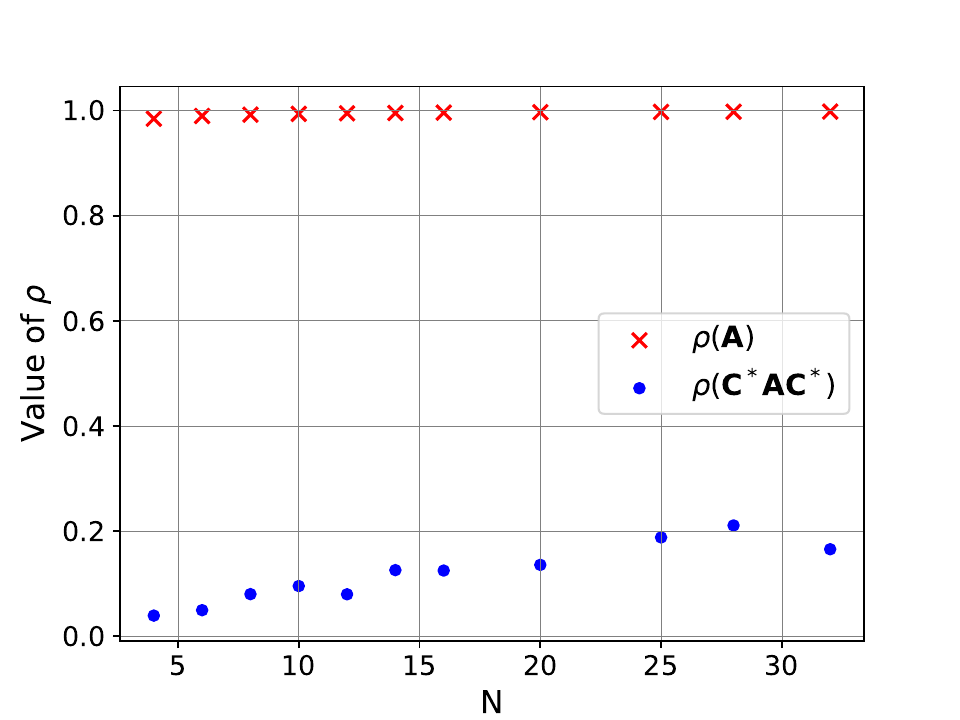}
\end{subfigure}
\caption{Left: Value of the condition numbers of the original and preconditioned linear systems (in semilog scale), for different values of $N$. Right: Value of CG convergence rate $\rho$ for both original and preconditioned linear system, for different values of $N$. As we can see, there is a significant improvement in the value of the condition number for the preconditioned system, and this translates into an improved value of $\rho$.}
\label{fig:condition-numbers}
\end{figure}

\begin{remark}
Note that the values of $N$ used in the above plots are relatively small compared to the ones encountered in practical applications. We could not go to higher values because of numerical bottlenecks. Indeed, computing $\Cmatrix^*$ first involves computing the eigendecomposition of $\Dmatrix$. To get the above plots, one then needs to compute the eigenvalues of the dense matrix $\Cmatrix^* \Amatrix \Cmatrix^*$. Even for $N = M = 100$, the size of the matrices $\Amatrix$, $\Dmatrix$ and $\Cmatrix^*$ is $29800 \times 29800$. This proved to be too slow in practice, and we believe that this section still illustrates well the benefits of our preconditioner.
\end{remark}

\begin{remark}
We are finally able to explain why it is numerically interesting to use quadratic penalties instead of hard constraints in problem~\eqref{prob:soft-constraints}. Indeed, it would be possible to derive the IRLS and CG algorithms of sections~\ref{sec:IRLS} and~\ref{sec:CGM} starting from problem~\eqref{prob:main-matrix-form}. The resulting vectorized linear system to be solved in the IRLS algorithm would then have the form
\begin{align*}
    \big((\I \kron \SN)^\top \diag{\myvec{\frac{1}{\Wv}}} ( \I \kron \SN) + (\TM \kron \I) \diag{\myvec{\frac{1}{\Wh}}}(\TM \kron \I)^\top\big) u = b',
\end{align*}
with $b' \in \R^n$ (we assume for simplicity that the weights $\Cv$ and $\Ch$ are unit weights). This results in the matrices $\Wv$ and $\Wh$ being `sandwiched' between two matrices, namely $( \I \kron \SN)$ and its transpose, and $(\TM \kron \I)$ and its transpose, respectively. As the IRLS algorithm makes progress, the linear system becomes more and more ill-conditioned, and we were not able to find good preconditioners to speed up the conjugate gradient iterates. In contrast, in matrix $\Amatrix$ defined in~\eqref{eq:matrix-soft-constraints}, the blocks in $\Wv$ and $\Wh$ and the block of the form $(\I \kron \SN^\top \SN) + (\TM\TM^\top \kron \I)$ are separate and form the diagonal, making the preconditioning much easier and efficient. A similar observation was made in~\cite{fornasier2016conjugate}, and we refer the interested reader to this work for more details on the numerical benefits of the quadratic penalties.

\end{remark}

\section{Experiments}
\label{sec:experiments}

We test our algorithm on both synthetic and real topographic data sets. The synthetic data allows us to verify that the output of our algorithm is accurate. For both synthetic and real data, we compare running time and solution quality of our method with the output of SNAPHU~\cite{chen2000network, chen2001two, chen2002phase}. SNAPHU is a widely used open-source software for phase unwrapping. The default algorithm is the statistical-cost, network-flow algorithm, which poses the phase unwrapping as a maximum a posteriori probability problem. SNAPHU also implements solvers for $L^p$-minimization phase unwrapping problem. In particular, the $L^1$-norm problem is solved using a modified network-simplex solver.

Whether SNAPHU solves the statistical-cost problems or $L^p$-norm problems, it is initialized with either a minimum cost flow (MCF) algorithm~\cite{costantini1998novel} or a minimum spanning tree (MST) algorithm~\cite{chen2000network}. It also offers the option to return the solution after only running the initialization. 

Note that MCF is another way of solving the $L^1$-norm problem. As noted in the SNAPHU documentation, the results of MCF and of the SNAPHU $L^1$-norm solver may be different, although in theory both should be $L^1$ optimal.

We shall therefore compare the IRLS algorithm against the five following SNAPHU options: MCF, MST, $L^1$-norm minimization with MCF initialization, $L^1$-norm minimization with MST initialization, statistical-cost with MCF initialization and statistical-cost with MST initialization, the latter being the default SNAPHU option.

IRLS experiments are ran on a Nvidia RTX8000 48 GB GPU, and the SNAPHU experiments on a Intel(R) Xeon(R) CPU E5-2650.

\subsection{Real Data Acquisition and Synthetic Data Generation}
\label{sec:data-details}

We now describe the Sentinel-1 InSAR dataset used in our experiments. For clarity, recall that the main geometric effects present in the interferogram between images $1$ and $2$ are the orbital $o_{12}$ and $\xi_{12}$ topographic components for which their unwrapped phase value can be approximated by \cite{Hanssen2002}:
\begin{equation}
    o_{12} = \dfrac{- 2 m \pi}{\lambda} B^\parallel_{12} \,\, \text{(rad)}, 
    \label{eq:orb_phase}
\end{equation}
\begin{equation}
    \xi_{12} = \dfrac{- 2 m \pi}{\lambda} \dfrac{B^\perp_{12} h}{R_{st}\sin \theta} \,\, \text{(rad)}, 
    \label{eq:topo_phase}
\end{equation}
where $\lambda$ is the radar carrier wavelength, $B^\parallel_{12}$ and $B^\perp_{12}$ correspond to the parallel and orthogonal baseline respectively (related to the separation of the two orbit passes in the direction parallel/orthogonal to the line of sight), $R_{st}$ is the satellite-target distance also known as the range, $\theta$ is the incidence angle and $m$ is equal to $2$ to account for the two-way wave propagation in Sentinel-1. A real Sentinel-1 interferogram also contains contributions related to the deformation between the two dates and the atmospheric propagation delays and is also impacted by phase noise and wrapping.

For the synthetic data generation, we are only interested in the topographic component $\xi_{12}$ because of the simplicity of the model given in~\eqref{eq:topo_phase}. Therefore, we use ten pairs of Sentinel-1 Synthetic Aperture Radar (SAR) images covering different locations in the world, chosen to exhibit variations in topography. To maximize the effect of topography $\xi_{12}$, we choose image pairs with large $B^\perp_{12}$, and we also minimize the potential deformation contributions by taking dates that are temporally close, as shown in Table \ref{tab:data_description}. We then adopt the standard simulation procedure of a differential InSAR processing chain to compute $o_{12}$ and $\xi_{12}$. In particular, we use the metadata of the two acquisitions in the image pair, along with their precise orbit auxiliary metadata, to obtain their respective camera models. The camera model allows us to compute the 2D position in the image of a 3D point. We use the camera model and the SRTM30~\cite{farr2007shuttle} Digital Elevation Model (DEM) to estimate $B^\parallel_{12}$, $B^\perp_{12}$, $h$, and $\theta$ for all pixels. Values of $\lambda$ and $R_{st}$ can be obtained by simple computations from the metadata, and~\eqref{eq:orb_phase} and~\eqref{eq:topo_phase} are used to get $o_{12}$ and $\xi_{12}$. Kayrros' toolbox for SAR data ("EOS-SAR") was used to perform the different computations. Note that when transforming the DEM into radar coordinates for estimating $h$, we adopt a mesh interpolation backgeocoding approach~\cite{Linde-Cerezo2021}. This induces some small interpolation artifacts in the form of a triangular mesh. Nevertheless, it is interesting to use $\xi_{12}$ for the synthetic data because it contains discontinuities related to terrain distortions due to the SAR acquisition geometry. In fact, besides phase noise, distortions such as layover and foreshortening are often cited as challenges for phase unwrapping~\cite{Hanssen2002}. 

On the other hand, after image alignment and burst stitching with the method described in \cite{akiki2022improved}, we compute the real interferograms compensated from the orbital component $o_{12}$ with: 
\begin{equation}
    z_{12} = z_1 \cdot z_2^* \cdot e^{-j o_{12}}, 
\end{equation}
where $z_1$ and $z_2$ correspond to each complex SAR image. The interferometric phase is simply obtained by 
\begin{equation}
    \phi_{12} = \operatorname{arg}(z_{12}), 
\end{equation}
where $\operatorname{arg}(.)$ denotes the argument of the complex number. As previously explained, our choice of images is such that $\phi_{12}$ is dominated by $\xi_{12}$. Therefore, $\phi_{12}$ is a real wrapped noisy phase containing fringes mainly related to topography. Since it is standard practice to denoise the interferometric phase prior to unwrapping  \cite{Hanssen2002, Braun2021}, we apply a Goldstein phase filter \cite{Goldstein1998} on the complex interferogram $z_{12}$, and the real denoised phase is given by 
\begin{equation}
    \widehat{\phi_{12}} = \operatorname{arg}(\widehat{z_{12}}), 
\end{equation}
where $\widehat{z_{12}}$ is the Goldstein filtered intereferogram. The Goldstein filter is a patch-based method, where each patch is denoised in the Fourier domain via: 
\begin{equation}
    Z'(u,v) = Z(u,v) \left\{\left(\mathbf{1}_{N_f} *|Z| \right)(u,v)\right\}^{\alpha}, 
\end{equation}
where $Z(u,v)$ and $Z'(u,v)$ correspond to the input and output patch Fourier transforms respectively, $\mathbf{1}_{N_f}$ is a uniform filter of size $N_f$ ($*$ is the convolution operation), and $\alpha$ is a factor between $0$ and $1$ that controls the amount of filtering. The patches are taken at a fixed step $s$, with a patch size equal to $4 \times s$, yielding an overlap of $75\%$. After the inverse Fourier transform of $Z'(u,v)$ is computed, the patches are recombined using a linear taper in the x and y dimensions. For our dataset, we set $\alpha = 1$, a step $s=16$ pixels, and a uniform filter size $N_f = 5$.  

The previously described process was repeated twice to obtain two datasets with different image sizes at full Sentinel-1 resolution. In both cases, we use the camera model to convert the centroid coordinates in Table \ref{tab:data_description} to the SAR image coordinates. For the first dataset, we compute crops of size $2048\times2048$ pixels centered on this location. For the second dataset, we compute the crops by stitching the full burst of the centroid with its previous and next burst in the same swath. This yields images of size about $ 4000 \times 20000$ pixels.

\begin{table}[t]
\centering
\begin{tabular}{l|c|c|c|c|c|c}
\toprule
Location & lon & lat & date 1 & date 2 & relorb & $B^\perp$(m) \\
\midrule
Arz & 36.5085 & 34.1601 & 20220615 & 20220627 & 21 & 100 \\
Etna & 14.6594 & 37.7404 & 20220411 & 20220505 & 124 & -138 \\
El Capitan & -120.0154 & 37.7357 & 20210710 & 20210728 & 144 & -99 \\
Kilimanjaro & 37.1166 & -2.9983 & 20180809 & 20180821 & 79 & 129 \\
Mount Sinai & 33.8923 & 28.5510 & 20230303 & 20230315 & 160 & 256 \\
Korab & 20.7207 & 41.8401 & 20220824 & 20220905 & 175 & -271 \\
Nevada & -119.2729 & 41.2154 & 20220518 & 20220530 & 144 & 156 \\
Zeil & 132.1604 & -23.3499 & 20211209 & 20211221 & 2 & -144 \\
Wulonggou & 96.0567 & 36.1646 & 20220824 & 20220905 & 172 & -270 \\
Warjan & 65.0716 & 32.4197 & 20230118 & 20230130 & 42 & -326 \\
\bottomrule
\end{tabular}
\caption{Dataset definition: for each location given by its central coordinates, we process the images at the two dates for a given relative orbit. Contains modified Copernicus Sentinel data.}
\label{tab:data_description}
\end{table}
\subsection{Weight computation}
\label{sec:weight-computation}

The weight matrices $\Cv$ and $\Ch$ are critical for the quality of the unwrapping. In classical SAR interferometry, the phase interferogram is usually provided along with the amplitudes $A_1$ and $A_2$ of the two acquired images, and a coherence map $C$. In our experiments, we use the statistical weights generated by SNAPHU, as they show very good practical performance and their computation is cheap compared to the running time of the algorithm. We therefore include the SNAPHU software in our distribution, and we slightly modify the original code to add the option of only computing the weights. For a detailed explanation of how those weights are generated, see~\cite{chen2001two, chen2001statistical}

For the experiments with real images, the weights are calculated using the interferogram, amplitudes $A_1$ and $A_2$, and the coherence map $C$. For the experiments with simulated images, only the interferogram is provided.

\subsection{Number of CG iterations and stopping criterion}
\label{sec:cg-update-strategy}

The conjugate gradient method presented in Section~\ref{sec:CGM} to solve the linear system~\eqref{eq:linear-system} requires an explicit maximum number of iterations. We seek to balance running time of the IRLS algorithm~\ref{algo-IRLS} with the quality of the output image. In particular, in early iterations of IRLS, it is often unnecessary to solve the linear system to a high precision to make significant progress on the value of the objective function, while higher precision is helpful in later iterations to refine image quality. This impacts the number of CG iterations at each IRLS step. One also needs to decide when to stop the IRLS algorithm as a whole. We propose a heuristic that aims at striking a good balance.

Let $m_{CG}(k)$ be the maximum number of iterations of CG allowed at iteration $k$ of the IRLS algorithm. We start by setting  $m_{CG}(0)$ to some predefined value.
For $k\geq 1$, after each updates of the weight matrices $\Wvk{k}, \Whk{k}$, we compute the relative improvement
\begin{align*}
    \Delta_{k-1, k} = \frac{H_\delta(\U^k, \Vvk{k}, \Vhk{k}, \Wvk{k-1}, \Whk{k-1}) - H_\delta(\U^k, \Vvk{k}, \Vhk{k}, \Wvk{k}, \Whk{k})}{H_\delta(\U^k, \Vvk{k}, \Vhk{k}, \Wvk{k-1}, \Whk{k-1})},
\end{align*}
where $H_\delta$ is as defined in Theorem~\ref{thm:CV-rate}. We then increase the maximum number of CG iterations based on whether or not we deem the relative improvement $\Delta_{k-1, k}$ significant enough, comparing it to a user-defined tolerance $\epsilon_{tol}$. If it is, then enough improvement was made by CG with the current number of maximum iterations, so we do not update $m_{CG}(k)$. If it is not, two cases arise. If the maximum number of CG iterations was already increased at the previous iteration, we consider that significant progress can no longer be made in the IRLS algorithm, so we stop and return $\U$. If the maximum number of CG iterations was not increased in the previous iteration, then it is possible that more CG iteration would lead to more progress, so we scale up $m_{CG}$ by a constant $\alpha > 1$. We summarize this as follows:
\begin{enumerate}
    \item If $\Delta_{k-1, k} > \epsilon_{tol}$, set $m_{CG}(k) = m_{CG}(k-1)$.
    \item If $\Delta_{k-1, k} \leq \epsilon_{tol}$ and $m_{CG}(k-1) \not = m_{CG}(k-2)$, stop and return $\U^k$.
    \item If $\Delta_{k-1, k} \leq \epsilon_{tol}$ and $m_{CG}(k-1) = m_{CG}(k-2)$, set $m_{CG}(k) = \alpha m_{CG}(k-1) $.
\end{enumerate}
In the following experiments we have set $m_{CG}(0) = 5,~\epsilon_{tol} = 10^{-3},~\alpha = 1.7.$

\subsection{Code details}
The full code can be downloaded from \url{https://github.com/bpauld/PhaseUnwrapping/tree/main}. In this section we give a quick overview of the main parts of our code. The file \textbf{final\_script.py} contains essentially the same steps. First we need a few imports.
\begin{lstlisting}
import numpy as np
from python_code.parameters import ModelParameters
from python_code.parameters import IrlsParameters
from python_code.unwrap import unwrap
\end{lstlisting}

\paragraph{Loading data}
Assuming the image to be unwrapped is contained in a .npy file called `X.npy', simply load it
\begin{lstlisting}
path_X = "X.npy" # replace with appropriate path
X = np.load(path_X) 
\end{lstlisting}

\paragraph{Defining weighting strategy}
The next step consists in defining $\Cv$ and $\Ch$. Several options are available. The first one is to provide user-defined weights.
\begin{lstlisting}
path_Cv = "Cv.npy" # replace with appropriate path
path_Ch = "Ch.npy" # replace with appropriate path
Cv = np.load(path_Cv)
Ch = np.load(path_Ch)
\end{lstlisting}
Alternatively, if the user does not wish to supply precomputed weights, the values of the weights should be set to `None', and the weighting strategy should be specified, as we will explain next.
\begin{lstlisting}
Cv, Ch = None, None
\end{lstlisting}
One can then decide to use constant unity weights via two equivalent ways.
\begin{lstlisting}
weighting_strategy = "uniform" 
weighting_strategy = None # last two lines are equivalent
\end{lstlisting}
Alternatively, one can use weights generated from SNAPHU, as mentioned in Section~\ref{sec:weight-computation}. The quality of the weights can be further improved if two amplitude files and a correlation file are provided, although this is not mandatory.
\begin{lstlisting}
weighting_strategy = "snaphu_weights"
path_amp1 = "amp1.npy" # replace with appropriate path
path_amp2 = "amp2.npy" # replace with appropriate path
path_corr = "corrfile.npy" # replace with appropriate path
amp1 = np.load(path_amp1)
amp2 = np.load(path_amp2)
corr = np.load(path_corrfile)
\end{lstlisting}
Note that if $\Cv$ and $\Ch$ are not set to \textbf{None}, their value will prevail over any weighting strategy specified and the algorithm will run with the supplied $\Cv$ and $\Ch$.

\paragraph{Defining model parameters}
The parameters of the model, which correspond to the value of $\tau$ as introduced in~\eqref{prob:soft-constraints} and the value of $\delta$ as introduced in~\eqref{eq:F_delta}, then need to be defined.
\begin{lstlisting}
model_params = ModelParameters()
model_params.tau = 1e-2
model_params.delta = 1e-6
\end{lstlisting}

\paragraph{Defining IRLS parameters}
Several parameters are necessary for the IRLS algorithm. Those are contained in the \textbf{IrlsParameters} class. The most important are related to the heuristics for updating the number of CG iterations described in section~\ref{sec:cg-update-strategy}.
\begin{lstlisting}
irls_params = IrlsParameters()
irls_params.max_iter_CG_strategy = "heuristics"
irls_params.max_iter_CG_start = 5
irls_params.rel_improvement_tol = 1e-3
irls_params.increase_CG_iteration_factor = 1.7
\end{lstlisting}
Other parameters can optionally be set, such as maximum number of IRLS iterations, maximum number of inner CG iterations, tolerance on the residue for stopping CG.

\paragraph{Unwrapping the image}
Finally, the unwrapping is done by calling the corresponding function.
\begin{lstlisting}
U, Vv, Vh = unwrap(X, 
            model_params=model_params,
            irls_params=irls_params,
            amp1=amp1,
            amp2=amp2,
            corr=corr,
            weighting_strategy=weighting_strategy,
            Cv=Cv, Ch=Ch)
\end{lstlisting}

\subsection{Results}

Recall that our model is invariant by a constant shift. To compute the error of our method, we therefore look for the shift that minimizes the norm between the original image and the shifted output. Namely, if the objective is an unwrapped image $\Xmatrix_u \in \R^{N \times M}$, then for an output $\U \in \R^{N \times M}$ produced by our algorithm, we compute the error $E(\U, \Xmatrix_u)$ as:
\begin{align*}
    E(\U, \Xmatrix_u) = \Xmatrix_u - (\U + \alpha \ones_{N \times M})
\end{align*}
where 
\begin{align*}
    \alpha = \text{argmin}_{\alpha \in \R} \left\{\norm{\Xmatrix_u - (\U +\alpha \ones_{N \times M}) }^2 \right\} = \frac{\sum_{\substack{1 \leq i \leq N \\ 1 \leq j \leq M }} (\Xmatrix_u - \U)_{i, j}}{ N M}.
\end{align*}

We start by exploring the performance of our proposed algorithm on simulated images such as the ones described in Section~\ref{sec:data-details}. This allows us to ensure that our method outputs an image which is consistent with the original simulated image, used as the ground truth to compute the error.

\begin{figure}[ht]
    \centering
    \includegraphics[width=0.33\linewidth]{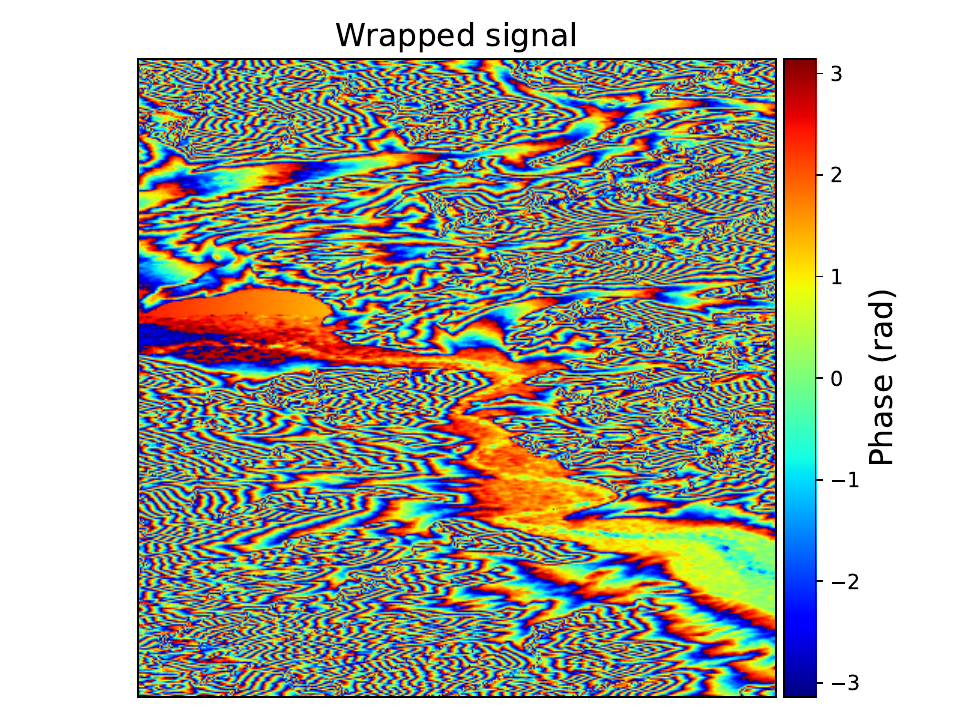}\hfil
    \includegraphics[width=0.33\linewidth]{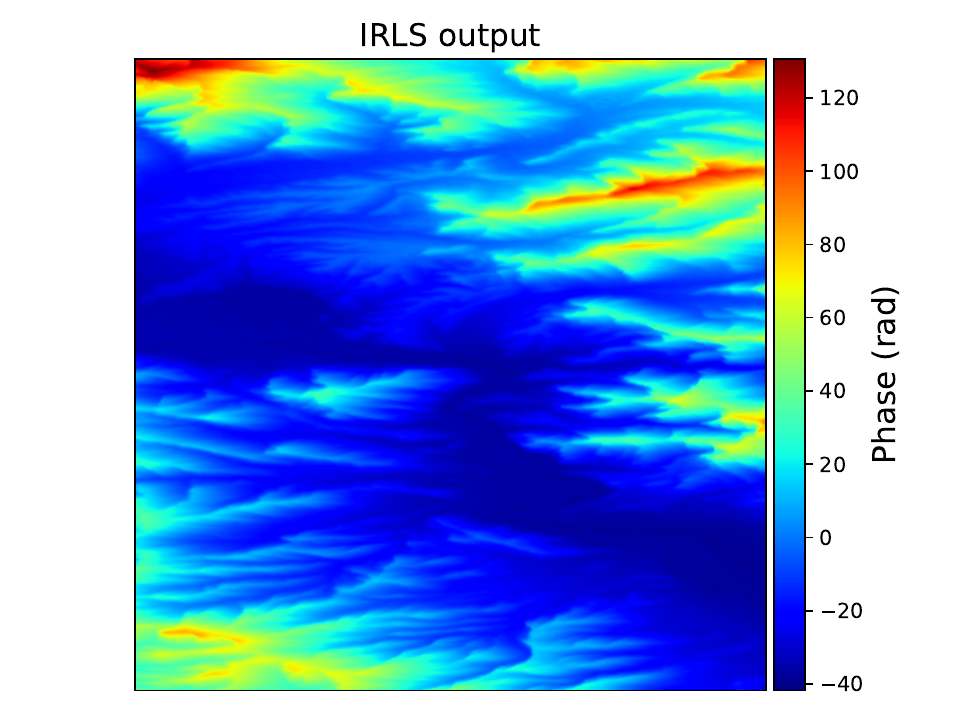}\hfil
    \includegraphics[width=0.33\linewidth]{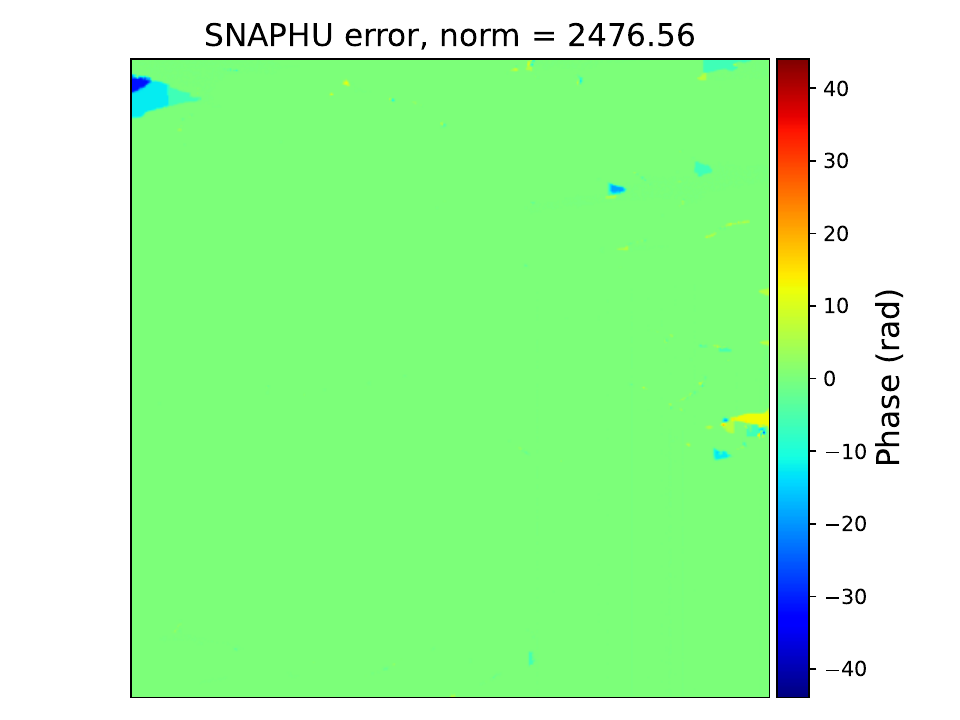}\par\medskip
    \includegraphics[width=0.33\linewidth]{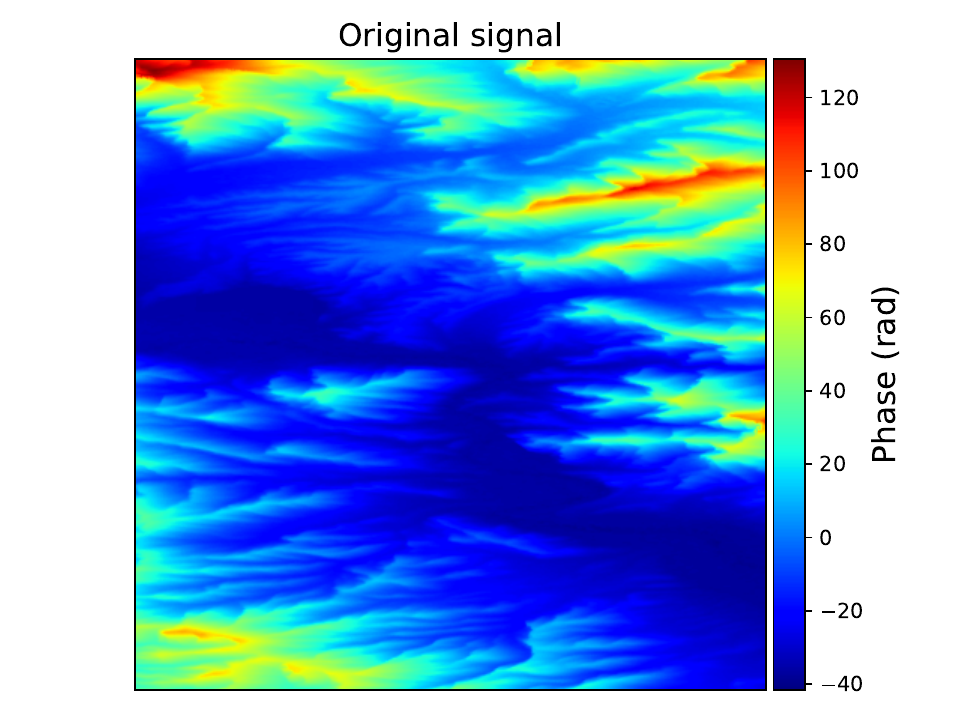}\hfil
    \includegraphics[width=0.33\linewidth]{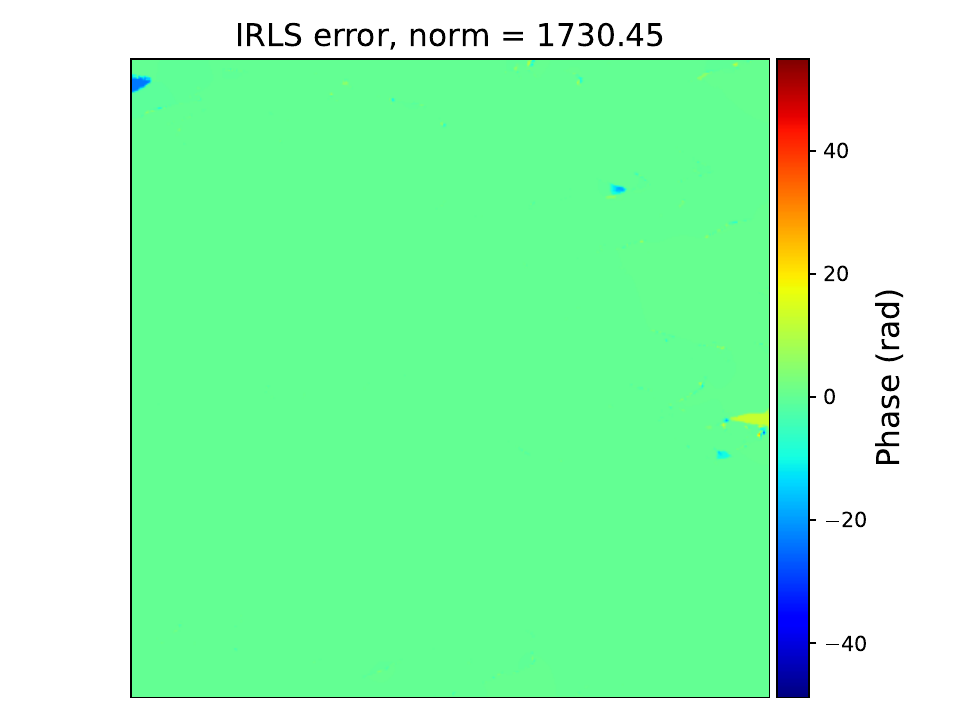}\hfil
    \includegraphics[width=0.33\linewidth]{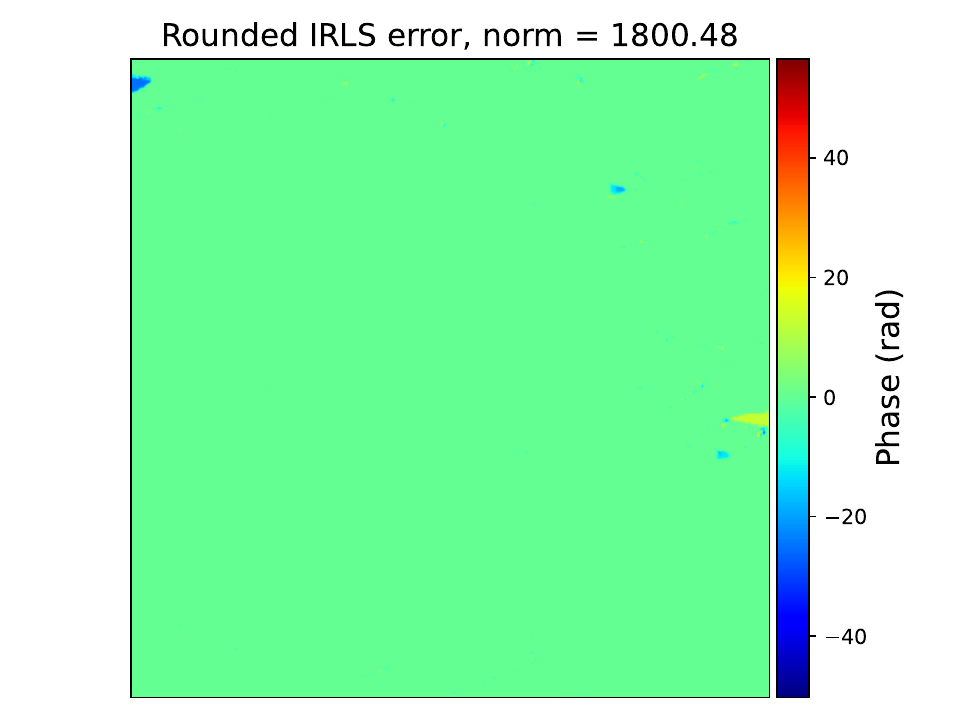}
    \caption{Comparison of the output of SNAPHU and IRLS on Warjan (Afghanistan) simulated image. Image size is $2048 \times 2048$.}
    \label{fig:noiseless-warjan-afghan}
\end{figure}



Figure~\ref{fig:noiseless-warjan-afghan} shows the output of our method for a specific location in Afghanistan. We observe that our method is able to reconstruct the original image almost correctly, and that the error made is comparable to the error made by the default SNAPHU method (MST + SC). We also explore the effect of rounding the IRLS output to enforce the output to be $2\pi$-congruent to the input image. We call this output \textit{rounded IRLS} in the plots. We observe that rounding does not seem to have significant effect on the performance of our method. Similar observations for different locations can be found in Appendix~\ref{sec:additional-experiments-simulated}.

On real images, Figure~\ref{fig:real-goldstein-warjan-afghan} shows that our algorithm is similar to the SNAPHU outputs. In many cases, our solution is closest to the $L^1$-SNAPHU solution. Moreover, we observe that the MST solution yields many artefacts. Similar observations for different locations can be found in Appendix~\ref{sec:additional-experiments-real}.

\begin{figure}[ht]
    \centering
    \includegraphics[width=0.25\linewidth]{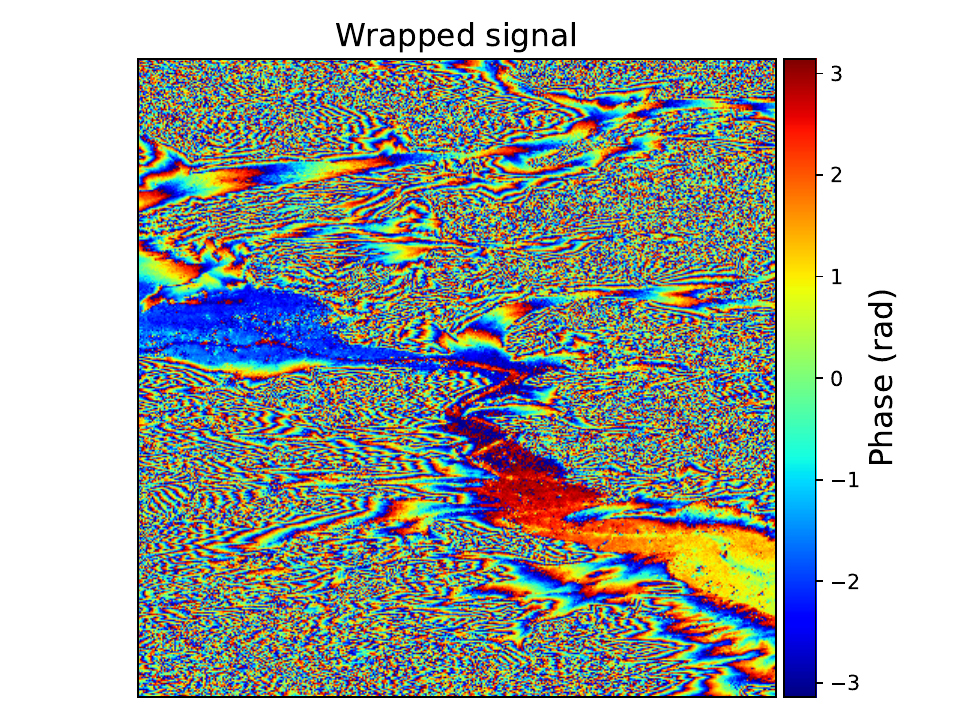}\hfil
    \includegraphics[width=0.25\linewidth]{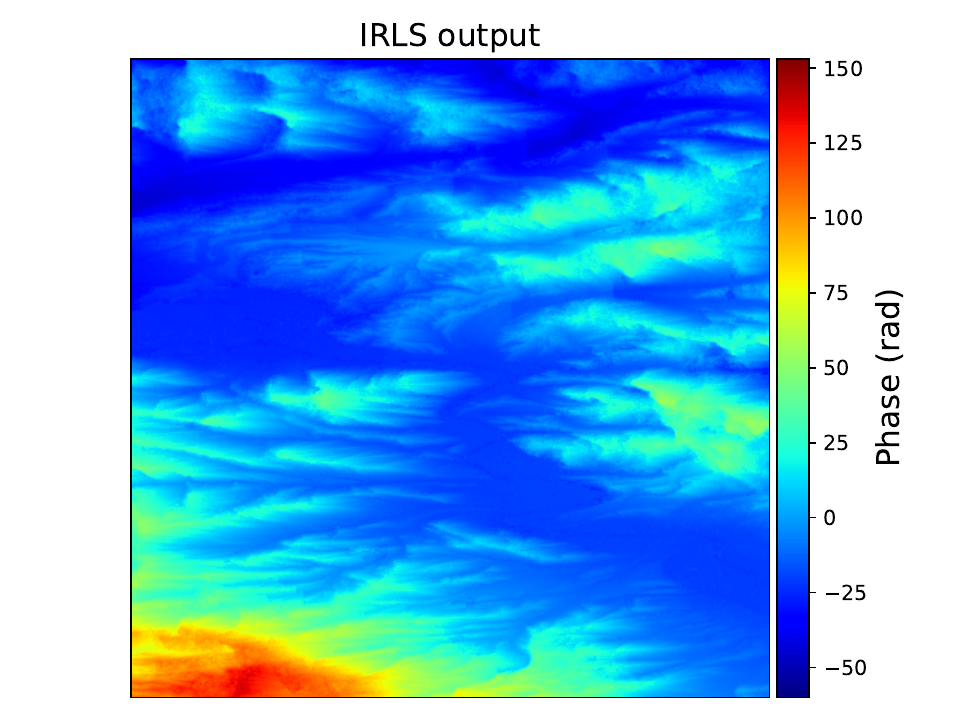}\hfil
    \includegraphics[width=0.25\linewidth]{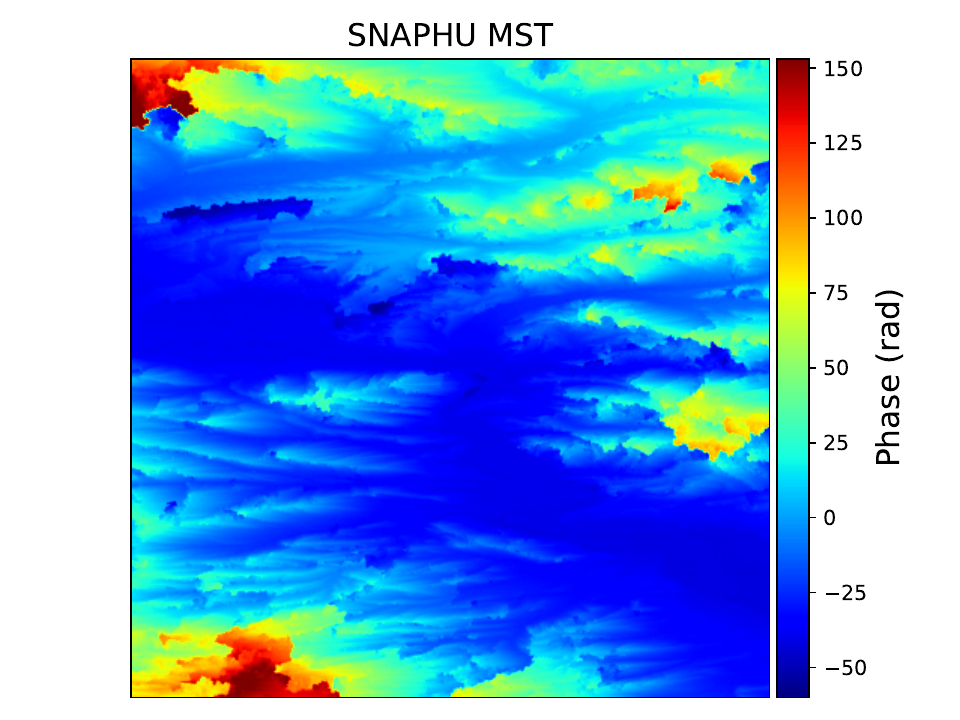}\hfil
    \includegraphics[width=0.25\linewidth]{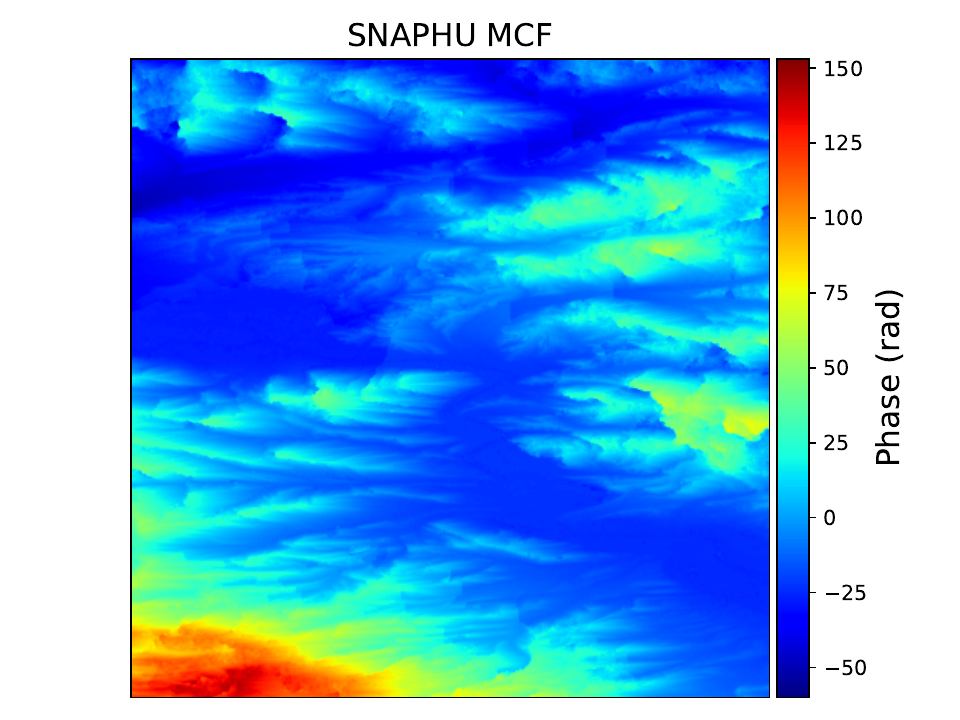}
    \par\medskip
    \includegraphics[width=0.25\linewidth]{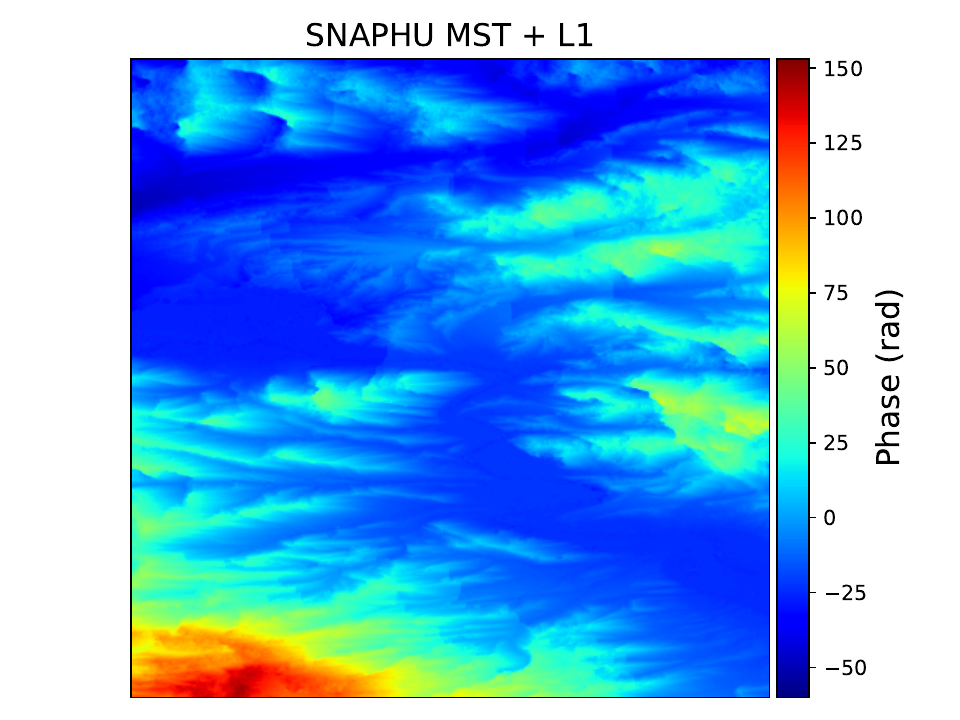}\hfil
    \includegraphics[width=0.25\linewidth]{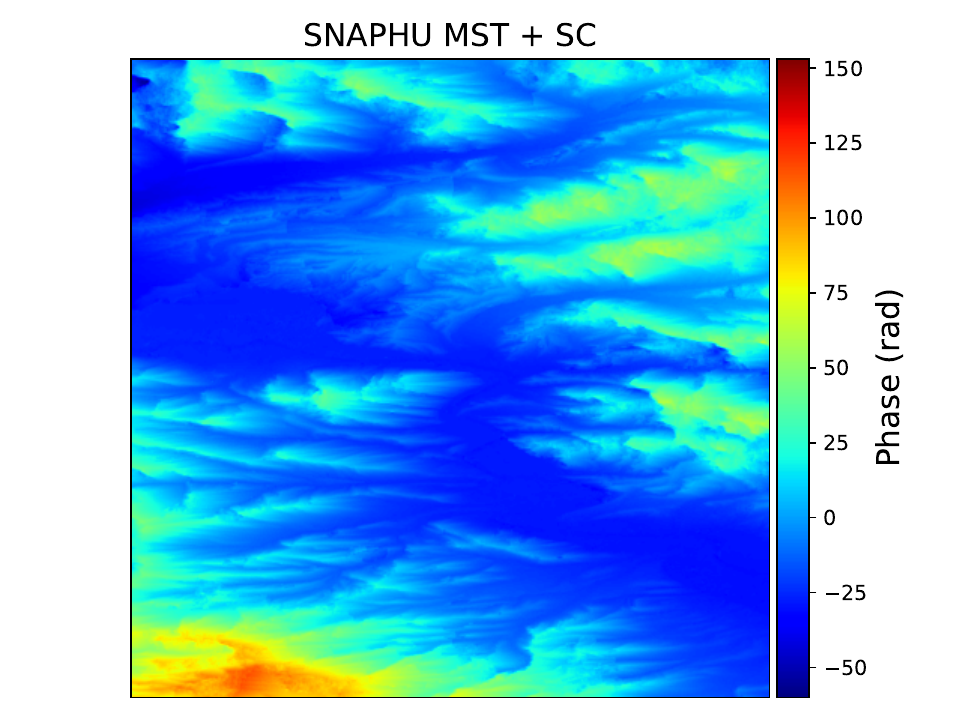}\hfil
    \includegraphics[width=0.25\linewidth]{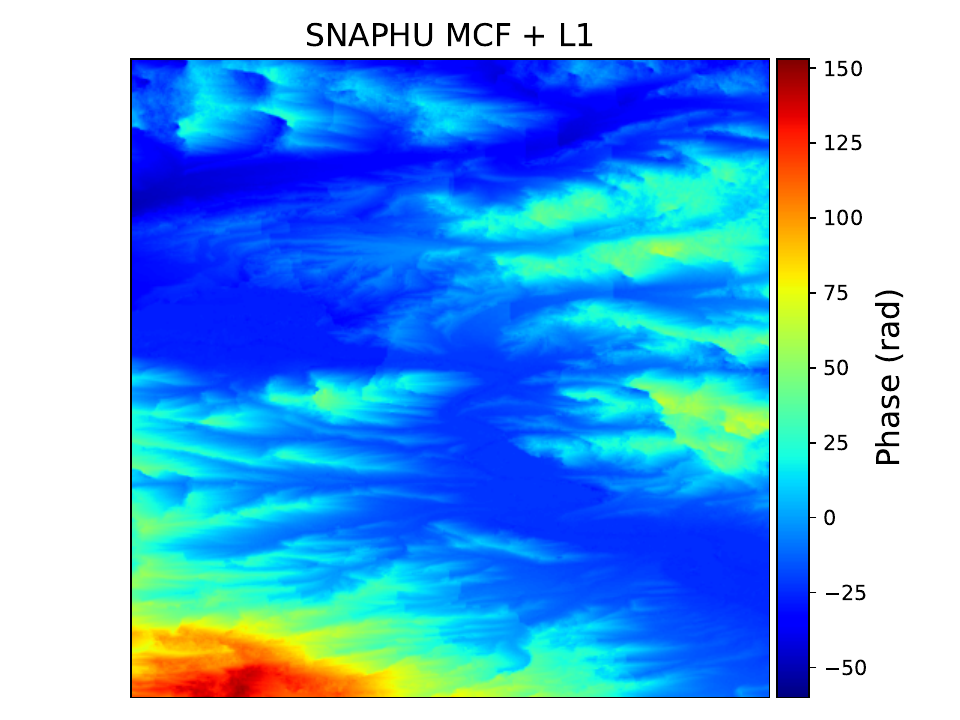}\hfil
    \includegraphics[width=0.25\linewidth]{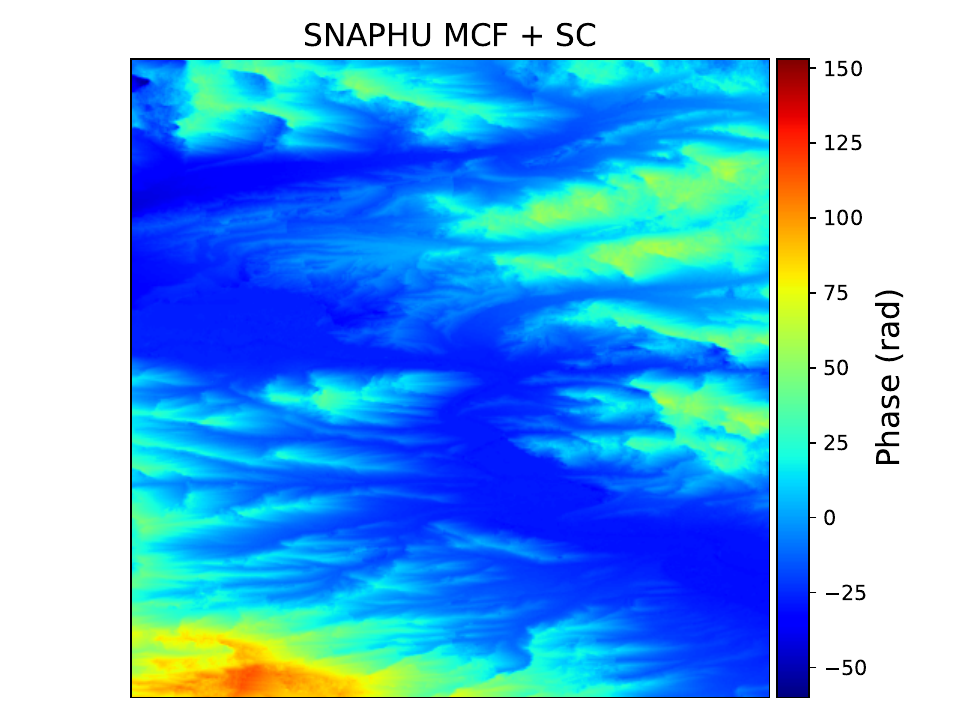}
    \caption{Comparison of the output of IRLS and of the different SNAPHU methods on Warjan (Afghanistan) image after Goldstein filtering. Image size is $2048 \times 2048$.}
    \label{fig:real-goldstein-warjan-afghan}
\end{figure}

We now turn our attention to the running time of our method in comparison with the different methods implemented in SNAPHU in Table~\ref{table:running-time-simulated} for simulated images and in Table~\ref{table:running-time-real-goldstein} for real images. We see that our algorithm significantly outperforms all other methods, except for the MST algorithm. However, MST yields poor-quality results on real images, and thus cannot be used as a full phase unwrapping method in practice, but only as an initialization procedure.

\begin{table}[ht]
\centering
\caption{Comparison of the running time (in seconds) of different phase unwrapping algorithms on {\bf simulated} $2048 \times 2048$ images. Our proposed method (IRLS) and the standard SNAPHU method (MST + SC) are in bold.}\label{table:running-time-simulated}
\begin{filecontents*}{figures/durations/durations_noiseless_v1.csv}
location,irls,mst,mcf,mstL1,mststat_cost,mcfL1,mcfstat_cost
Arz,13.36,3.96,20.62,43.81,44.33,60.18,61.1
Etna,24.12,3.83,96.42,41.46,42.93,133.92,134.88
El Capitan,24.2,3.83,88.12,51.86,67.94,123.25,129.6
Korab,31.81,3.87,117.69,47.11,275.84,145.38,238.41
Kilimanjaro,22.68,3.94,95.86,44.22,44.34,134.04,139.04
Mount Sinai,30.67,3.96,92.87,26.02,32.05,112.53,120.94
Nevada,7.13,3.23,16.23,42.95,43.76,56.69,58.6
Zeil,11.12,3.81,20.68,39.03,39.65,55.48,58.22
Wulonggou,32.0,4.32,116.85,31.22,194.32,131.86,187.13
Warjan,29.33,4.22,112.37,22.0,64.49,127.0,172.47
\textbf{Mean},22.64,3.9,77.77,38.97,84.96,108.03,130.04
\textbf{Std},8.63,0.27,39.55,9.06,77.92,34.07,56.57
\end{filecontents*}
\csvreader[tabular = l|c|c|c|c|c|c|c, 
table head = \hline Location & \textbf{IRLS} & MST & MCF & MST + L1 & \textbf{MST + SC} & MCF + L1  & MCF + SC\\\hline,
late after last line=\\\hline
]{figures/durations/durations_noiseless_v1.csv}{}{%
\ifnum\thecsvrow=11 \hline\fi
\csvcoli & \textbf{\csvcolii} & \csvcoliii &\csvcoliv &\csvcolv& \textbf{\csvcolvi} & \csvcolvii &\csvcolviii}
\end{table}

\begin{table}[ht]
\centering
\caption{Comparison of the running time (in seconds) of different phase unwrapping algorithms on {\bf real} $2048 \times 2048$ images after Goldstein filtering. Our proposed method (IRLS) and the standard SNAPHU method (MST + SC) are in bold.}\label{table:running-time-real-goldstein}
\begin{filecontents*}{figures/durations/durations_real_goldstein_v2.csv}
location,irls,mst,mcf,mstL1,mststat_cost,mcfL1,mcfstat_cost
Arz,22.98,5.15,102.04,18.59,18.82,113.36,121.71
Etna,24.15,5.36,186.4,208.74,513.65,224.15,264.87
El Capitan,23.26,5.6,195.62,117.56,146.91,217.23,250.83
Korab,22.78,5.61,212.11,771.01,1147.65,325.68,384.24
Kilimanjaro,25.37,5.39,183.11,156.91,231.65,207.91,241.84
Mount Sinai,17.75,5.33,146.41,445.75,844.97,266.16,342.23
Nevada,16.38,5.22,113.12,14.75,17.88,121.66,131.02
Zeil,22.82,5.26,140.7,22.74,36.67,145.42,158.33
Wulonggou,21.3,5.4,226.34,1130.57,2833.8,289.99,796.35
Warjan,21.21,5.18,183.67,826.21,1190.88,295.21,540.06
\textbf{Mean},21.8,5.35,168.95,371.28,698.29,220.68,323.15
\textbf{Std},2.65,0.15,39.36,381.91,833.35,71.15,198.74
\end{filecontents*}
\csvreader[tabular = l|c|c|c|c|c|c|c, table head = \hline Location & \textbf{IRLS} & MST & MCF & MST + L1 & \textbf{MST + SC} & MCF + L1  & MCF + SC\\\hline,
late after last line=\\\hline
]{figures/durations/durations_real_goldstein_v2.csv}{}{%
\ifnum\thecsvrow=11 \hline\fi
\csvcoli & \textbf{\csvcolii} & \csvcoliii &\csvcoliv &\csvcolv& \textbf{\csvcolvi} & \csvcolvii &\csvcolviii}
\end{table}

To end this section, we study how the running time scales with the size of the image, for our algorithm and SNAPHU (MST + SC). We unwrap images of size $4000 \times M$ for different values of $M$ ranging from $2000$ to $16000$. We do so for both simulated images and Goldstein filtered real images. We use 10 simulated images and 4 real images. We only use 4 locations for real images because for the remaining ones, it was not possible to find large areas of high coherence, making the unwrapping impossible. Results are displayed in Figure~\ref{fig:scaling}. We not only observe a significant speedup in computing time, but also that our method scales better with the size of problem than SNAPHU.

\begin{figure}[ht]
    \centering
    \includegraphics[width=0.5\linewidth]{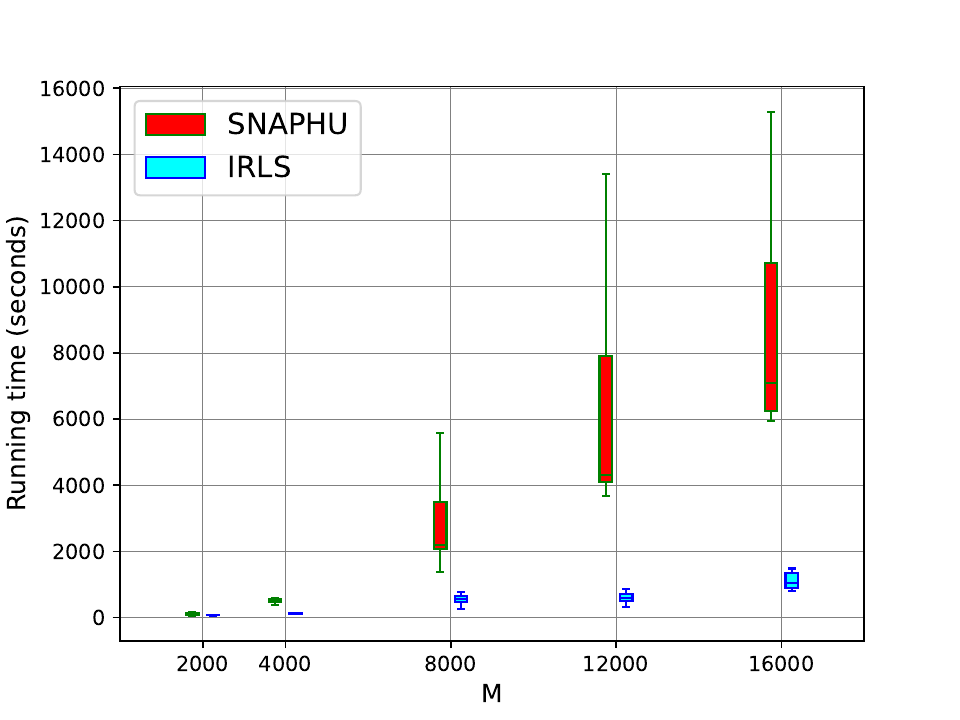}\hfil
    \includegraphics[width=0.5\linewidth]{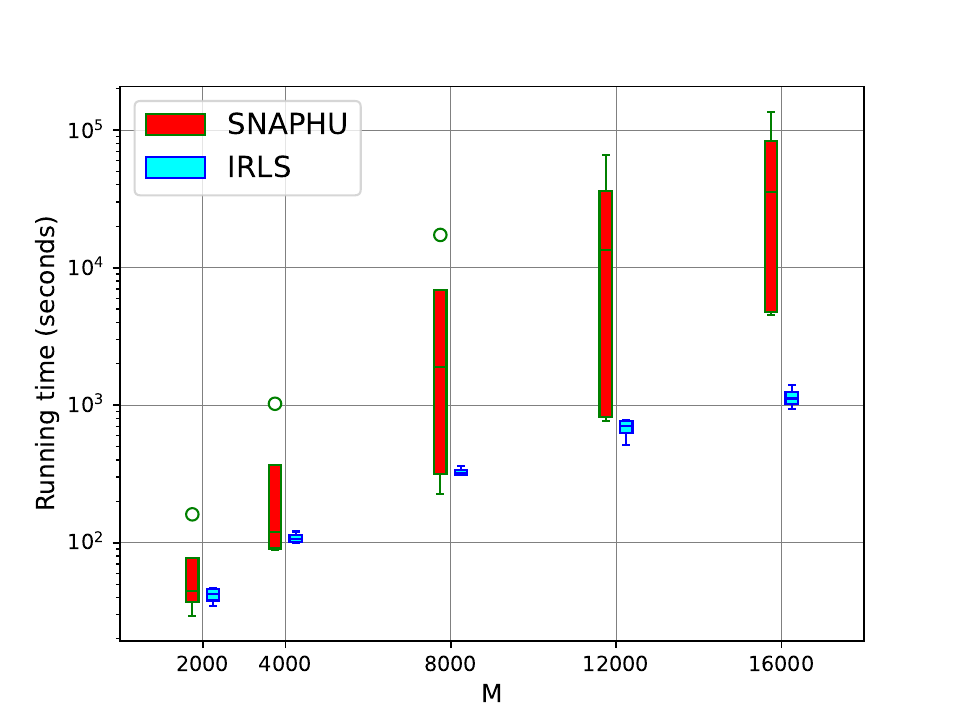}
    \caption{Box plot of the running times of IRLS and SNAPHU (MST + SC) on images of size $4000 \times M$ for different values of $M$. Left: simulated images. Right: Goldstein filtered real images (log scale).}
\label{fig:scaling}
\end{figure}




\section{Conclusion}\label{sec13}
In this work we studied the 2D phase unwrapping $L^1$-norm minimization problem. We proposed an iteratively reweighted least squares algorithm which, combined with an efficient conjugate gradient method, allows to solve the problem using only simple linear algebra operations. This led to a GPU-compatible algorithm, which we implemented and for which we showed competitive performance compared to commonly-used phase unwrapping methods. In the future, we hope to explore how the IRLS approach can be used with tiling strategies, to further accelerate the unwrapping on very large images~\cite{chen2002phase}.

\section{Acknowledgements}

The authors are grateful to Justin Carpentier and Alessandro Rudi for fruitful discussions.

Benjamin Dubois-Taine acknowledges support from the European Research Council (grant SEQUOIA 724063).
This work was funded in part by the french government under management of Agence Nationale
de la recherche as part of the “Investissements d’avenir” program, reference ANR-19-P3IA-0001
(PRAIRIE 3IA Institute) and a Google focused award. Roland Akiki acknowledges support from ANRT (grant N$^{\circ}$ 2019/2003).

The authors are grateful to the CLEPS infrastructure from the Inria of Paris for providing resources and support.

\backmatter

\begin{appendices}

\section{Matrix visualization}\label{sec:matrix-definition}

\begin{align*}
        \I_M \kron \SN &= \begin{pmatrix}
            -1 &\ 1 &\  &\ &\\\
            0 &\ -1 &\ 1 &\ &\ \\
            &\ 0 &\ -1  & \ddots&\ \\
            &\ &\ \ddots & \ddots &\ 1 &\ \\
            &\ &\ &\ 0 & -1 &\ 1
        \end{pmatrix} \in \R^{M(N-1) \times MN},\\
       \I_M \kron \SN^\top \SN &=  \begin{pmatrix}
            \SN^\top \SN &\ &\ \\
            &\ \ddots &\ \\
            &\ &\ \SN^\top \SN
        \end{pmatrix}\in \R^{MN \times MN},
\intertext{with}
         \SN^\top \SN &= \begin{pmatrix}
            1 &\ -1 &\  &\ &\\\
            -1 &\ 2 &\ -1 &\ &\ \\
            &\ -1 &\ 2  & \ddots&\ \\
            &\ &\ \ddots & \ddots &\ -1 &\ \\
            &\ &\ &\ -1 & 2 &\ -1 \\
            &\ &\ &\ &\ -1 & 1
        \end{pmatrix} \in \R^{N \times N},
\intertext{and}
        \TM\TM^\top \kron \I_N &= \begin{pmatrix}
            \I_N &\ -\I_N &\  &\ &\\\
            -\I_N &\ 2\I_N &\ -\I_N &\ &\ \\
            &\ -\I_N &\ 2\I_N  & \ddots&\ \\
            &\ &\ \ddots & \ddots &\ -\I_N &\ \\
            &\ &\ &\ -\I_N & 2\I_N &\ -\I_N \\
            &\ &\ &\ &\ -\I_N & \I_N
        \end{pmatrix} \in \R^{MN \times MN},\\
        \TM \kron \I_N &= \begin{pmatrix}
            -\I_N &\   &\ &\\\
            \I_N &\ -\I_N  &\ &\ \\
            &\ \I_N &\ -\I_N  &\ \\
            &\ &\ \ddots & \ddots &\ &\ \\
            &\ &\ &\ \I_N & -\I_N\\
            &\ &\ &\ &\ \I_N
        \end{pmatrix} \in \R^{MN \times (M-1)N}
    \end{align*}

\section{Proof of Proposition~\ref{prop:equivalence-precond}}
\label{sec:proof-equivalence-precond}
\equivalencePrecond*
    \begin{proof}
    We write plain letters for the iterates of CG to solve~\eqref{eq:linear-system-PCG} with preconditioner $\Dmatrix$, and use tilded letters for the iterates of $CG$ without a preoconditioner to solve~\eqref{eq:equivalent-linear-system-PCG} starting from $\tilde{x}_0 = \Cmatrix x_0$. We claim that $\tilde{x}_l = \Cmatrix^*x_l$ for all $l \in \mathbb{N}$.

    We have
    \begin{align*}
        r_0 &= b - \Amatrix x_0,\\
        \tilde{r}_0 &= \Cmatrix^*b - \Cmatrix^* A \Cmatrix^* \tilde{x}_0 = \Cmatrix^*(b - \Amatrix \Cmatrix^* \Cmatrix x_0) = \Cmatrix^* (b - \Amatrix x_0) = \Cmatrix^* r_0
    \end{align*}
    where the second to last inequality comes from the fact that $\Amatrix\Cmatrix^* \Cmatrix = \Amatrix$, which holds because $\Cmatrix^* \Cmatrix = \sum_{i=2}^n d_i d_i^\top = \I_n - d_1 d_1^\top$ and the fact that $\Amatrix d_1 = 0$.

     Since by construction $b$ is in the range of $\Amatrix$, we get that $r_0$ is in the range of $\Amatrix$ (this is because in our case, $\Amatrix$ and $b$ come from the minimization of a least-squares problem). Since $q_1$ is the nullspace of $\Amatrix$, we get that $r_0^\top q_1 = r_0^\top d_1 = 0$ and the system $\Dmatrix z_0 = r_0$ is well-defined. One solution, which is the one we choose in practice, is to set
    \begin{align*}
        z_0 = \left(\sum_{i=2}^n \frac{1}{\gamma_i} d_i d_i^\top \right) r_0 = \Cmatrix^* \Cmatrix^* r_0.
    \end{align*}
    Moreover,
    \begin{align*}
        \tilde{z}_0 = \tilde{r}_0 = \Cmatrix^* r_0 = \Cmatrix z_0
    \end{align*}
    since $\Cmatrix z_0 = \Cmatrix \Cmatrix^* \Cmatrix^* r_0$ and, as above, $\Cmatrix\Cmatrix^* \Cmatrix^* = \Cmatrix^*$, which holds because $\Cmatrix\Cmatrix^* \Cmatrix^* = ( \I_n - d_1 d_1^\top) \Cmatrix^*$ and the fact that $d_1^\top \Cmatrix^* = 0$.
    Then,
    \begin{align*}
        p_0 &= z_0\\
        \tilde{p}_0 &= \tilde{z}_0 = \Cmatrix^* r_0 =  \Cmatrix z_0=\Cmatrix p_0\\
        \rho_0 & = z_0^\top r_0 \\
        \tilde{\rho}_0 &= \tilde{z}_0^\top \tilde{r}_0 = z_0 \Cmatrix \Cmatrix^* r_0 = z_0^\top r_0 = \rho_0,
        \end{align*}
        since $r_0$ is in the range of $\Amatrix$, and thus of $\Cmatrix\Cmatrix^*$, and thus that $\Cmatrix\Cmatrix^*r_0 = r_0$. Then 
        \begin{align*}
        \alpha_0 &= \frac{\rho_0}{p_0^\top \Amatrix^\top p_0}\\       
        \tilde{\alpha}_0 &= \frac{\tilde{\rho}_0 }{ \tilde{p}_0^\top \Cmatrix^* \Amatrix \Cmatrix^* \tilde{p}_0} = \frac{\rho_0}{p_0^\top \Cmatrix\Cmatrix^* \Amatrix \Cmatrix\Cmatrix^* p_0}.
    \end{align*}
    As above we have that $p_0^\top \Cmatrix\Cmatrix^* \Amatrix \Cmatrix\Cmatrix^* p_0 = p_0^\top \Amatrix p_0$ so that $\tilde{\alpha}_0 = \alpha_0$. Then,
    \begin{align*}
        x_1 &= x_0 + \alpha_0 p_0,\\
        \tilde{x}_1 &= \tilde{x}_0 + \alpha_0 \tilde{p}_0= \Cmatrix x_0 + \alpha_0 \Cmatrix p_0 = \Cmatrix x_1.
    \end{align*}
    Then
    \begin{align*}
        r_1 &= r_0 - \alpha_0 \Amatrix p_0\\
        \tilde{r}_1 &= \tilde{r}_0 - \alpha_0 \Cmatrix^* \Amatrix \Cmatrix^* \tilde{p}_0 = \Cmatrix^* (r_0 - \alpha_0 \Amatrix \Cmatrix^* \Cmatrix p_0) = \Cmatrix^* r_1.
    \end{align*}
    In particular, we get that $r_1$ is in the range of $\Amatrix$, and thus of $\Dmatrix$. Therefore the system $\Dmatrix z_1 = r_1$ is well-defined and we choose in practice to set
    \begin{align*}
        z_1 = \left( \sum_{i=2}^n \frac{1}{\gamma_i} d_i d_i^\top\right) r_1 = \Cmatrix^* \Cmatrix^* r_1.
    \end{align*}
    Moreover, $\tilde{z}_1 = \tilde{r}_1 = \Cmatrix^*r_1 = \Cmatrix z_1$ by the same argument as above. Then
    \begin{align*}
        \rho_{1} &= r_{1}^\top z_1 \\      \tilde{\rho}_{1} &= \tilde{r}_1^\top \tilde{z}_1 = r_1^\top \Cmatrix^* \Cmatrix z_1 = r_1^\top z_1 = \rho_1,
    \end{align*}
    since again $\Cmatrix^* \Cmatrix = \I_n - d_1 d_1^\top$ ad $r_1$ is in the range of $\Amatrix$ and thus $r_1^\top d_1 = 0$.
    Thus $\beta = \tilde{\beta}$ and 
    \begin{align*}
        p_1 &= z_1 + \beta p_0 \\
        \tilde{p}_1 &= \tilde{z}_1 + \beta \tilde{p}_0 = \Cmatrix p_1
    \end{align*}
 One can then apply the same reasoning at the next loop $l=1$ and this shows that the iterates always satisfy $\tilde{x}_l = \Cmatrix x_l$. 
    \end{proof}

\section{Additional experiments}
\subsection{Simulated images}

\label{sec:additional-experiments-simulated}

\foreach \location/\titletext in {arz_lebanon/Arz Lebanon , etna/Etna , elcapitan/El Capitan , kilimanjaro/Kilimanjaro , mount_sinai/Mount Sinai , korab_northmacedonia/Korab , nevada_usa/Nevada , zeil_australia/Zeil (Australia) , wulonggou_china/Wulonggou (China) }{
    \begin{figure}[H]
    \centering
    \includegraphics[width=0.33\linewidth]{figures/plots/\location/noiseless_v1_wrapped_signal.pdf}\hfil
    \includegraphics[width=0.33\linewidth]{figures/plots/\location/noiseless_v1_IRLS_output.pdf}\hfil
    \includegraphics[width=0.33\linewidth]{figures/plots/\location/noiseless_v1_snaphu_error.pdf}\par\medskip
    \includegraphics[width=0.33\linewidth]{figures/plots/\location/noiseless_v1_original_signal.pdf}\hfil
    \includegraphics[width=0.33\linewidth]{figures/plots/\location/noiseless_v1_IRLS_error.pdf}\hfil
    \includegraphics[width=0.33\linewidth]{figures/plots/\location/noiseless_v1_rounded_IRLS_error.pdf}
    \caption{Comparison of the output of SNAPHU (MST + SC) and IRLS on \titletext simulated image. Image size is $2048 \times 2048$.}
\end{figure}
}

\subsection{Real images with Goldstein filtering}
\label{sec:additional-experiments-real}

\foreach \location/\titletext in {arz_lebanon/Arz Lebanon , etna/Etna , elcapitan/El Capitan , kilimanjaro/Kilimanjaro , mount_sinai/Mount Sinai , korab_northmacedonia/Korab , nevada_usa/Nevada , zeil_australia/Zeil (Australia) , wulonggou_china/Wulonggou (China) }{
    \begin{figure}[H]
    \centering
    \includegraphics[width=0.25\linewidth]{figures/plots/\location/real_goldstein_v2_real_goldstein_signal.pdf}\hfil
    \includegraphics[width=0.25\linewidth]{figures/plots/\location/real_goldstein_v2_irls_output.pdf}\hfil
    \includegraphics[width=0.25\linewidth]{figures/plots/\location/real_goldstein_v2_snaphu_mst.pdf}\hfil
    \includegraphics[width=0.25\linewidth]{figures/plots/\location/real_goldstein_v2_snaphu_mcf.pdf}
    \par\medskip
    \includegraphics[width=0.25\linewidth]{figures/plots/\location/real_goldstein_v2_snaphu_mstL1.pdf}\hfil
    \includegraphics[width=0.25\linewidth]{figures/plots/\location/real_goldstein_v2_snaphu_mststat_cost.pdf}\hfil
    \includegraphics[width=0.25\linewidth]{figures/plots/\location/real_goldstein_v2_snaphu_mcfL1.pdf}\hfil
    \includegraphics[width=0.25\linewidth]{figures/plots/\location/real_goldstein_v2_snaphu_mcfstat_cost.pdf}
    \caption{Comparison of the output of IRLS and of the different SNAPHU methods on \titletext real image after Goldstein filtering. Image size is $2048 \times 2048$.}
\end{figure}
}

\end{appendices}


\bibliography{sn-bibliography}


\begin{thebibliography}{56}
\ifx \bisbn   \undefined \def \bisbn  #1{ISBN #1}\fi
\ifx \binits  \undefined \def \binits#1{#1}\fi
\ifx \bauthor  \undefined \def \bauthor#1{#1}\fi
\ifx \batitle  \undefined \def \batitle#1{#1}\fi
\ifx \bjtitle  \undefined \def \bjtitle#1{#1}\fi
\ifx \bvolume  \undefined \def \bvolume#1{\textbf{#1}}\fi
\ifx \byear  \undefined \def \byear#1{#1}\fi
\ifx \bissue  \undefined \def \bissue#1{#1}\fi
\ifx \bfpage  \undefined \def \bfpage#1{#1}\fi
\ifx \blpage  \undefined \def \blpage #1{#1}\fi
\ifx \burl  \undefined \def \burl#1{\textsf{#1}}\fi
\ifx \doiurl  \undefined \def \doiurl#1{\url{https://doi.org/#1}}\fi
\ifx \betal  \undefined \def \betal{\textit{et al.}}\fi
\ifx \binstitute  \undefined \def \binstitute#1{#1}\fi
\ifx \binstitutionaled  \undefined \def \binstitutionaled#1{#1}\fi
\ifx \bctitle  \undefined \def \bctitle#1{#1}\fi
\ifx \beditor  \undefined \def \beditor#1{#1}\fi
\ifx \bpublisher  \undefined \def \bpublisher#1{#1}\fi
\ifx \bbtitle  \undefined \def \bbtitle#1{#1}\fi
\ifx \bedition  \undefined \def \bedition#1{#1}\fi
\ifx \bseriesno  \undefined \def \bseriesno#1{#1}\fi
\ifx \blocation  \undefined \def \blocation#1{#1}\fi
\ifx \bsertitle  \undefined \def \bsertitle#1{#1}\fi
\ifx \bsnm \undefined \def \bsnm#1{#1}\fi
\ifx \bsuffix \undefined \def \bsuffix#1{#1}\fi
\ifx \bparticle \undefined \def \bparticle#1{#1}\fi
\ifx \barticle \undefined \def \barticle#1{#1}\fi
\bibcommenthead
\ifx \bconfdate \undefined \def \bconfdate #1{#1}\fi
\ifx \botherref \undefined \def \botherref #1{#1}\fi
\ifx \url \undefined \def \url#1{\textsf{#1}}\fi
\ifx \bchapter \undefined \def \bchapter#1{#1}\fi
\ifx \bbook \undefined \def \bbook#1{#1}\fi
\ifx \bcomment \undefined \def \bcomment#1{#1}\fi
\ifx \oauthor \undefined \def \oauthor#1{#1}\fi
\ifx \citeauthoryear \undefined \def \citeauthoryear#1{#1}\fi
\ifx \endbibitem  \undefined \def \endbibitem {}\fi
\ifx \bconflocation  \undefined \def \bconflocation#1{#1}\fi
\ifx \arxivurl  \undefined \def \arxivurl#1{\textsf{#1}}\fi
\csname PreBibitemsHook\endcsname

\bibitem[\protect\citeauthoryear{Graham}{1974}]{graham1974synthetic}
\begin{barticle}
\bauthor{\bsnm{Graham}, \binits{L.C.}}:
\batitle{Synthetic interferometer radar for topographic mapping}.
\bjtitle{Proceedings of the IEEE}
\bvolume{62}(\bissue{6}),
\bfpage{763}--\blpage{768}
(\byear{1974})
\end{barticle}
\endbibitem

\bibitem[\protect\citeauthoryear{Zebker and Goldstein}{1986}]{zebker1986topographic}
\begin{barticle}
\bauthor{\bsnm{Zebker}, \binits{H.A.}},
\bauthor{\bsnm{Goldstein}, \binits{R.M.}}:
\batitle{Topographic mapping from interferometric synthetic aperture radar observations}.
\bjtitle{Journal of Geophysical Research: Solid Earth}
\bvolume{91}(\bissue{B5}),
\bfpage{4993}--\blpage{4999}
(\byear{1986})
\end{barticle}
\endbibitem

\bibitem[\protect\citeauthoryear{Ghiglia and Pritt}{1998}]{ghiglia1998two}
\begin{botherref}
\oauthor{\bsnm{Ghiglia}, \binits{D.C.}},
\oauthor{\bsnm{Pritt}, \binits{M.D.}}:
Two-dimensional phase unwrapping: theory, algorithms, and software.
Wiely-Interscience, first ed.(April 1998)
(1998)
\end{botherref}
\endbibitem

\bibitem[\protect\citeauthoryear{Rosen et~al.}{2000}]{rosen2000synthetic}
\begin{barticle}
\bauthor{\bsnm{Rosen}, \binits{P.A.}},
\bauthor{\bsnm{Hensley}, \binits{S.}},
\bauthor{\bsnm{Joughin}, \binits{I.R.}},
\bauthor{\bsnm{Li}, \binits{F.K.}},
\bauthor{\bsnm{Madsen}, \binits{S.N.}},
\bauthor{\bsnm{Rodriguez}, \binits{E.}},
\bauthor{\bsnm{Goldstein}, \binits{R.M.}}:
\batitle{Synthetic aperture radar interferometry}.
\bjtitle{Proceedings of the IEEE}
\bvolume{88}(\bissue{3}),
\bfpage{333}--\blpage{382}
(\byear{2000})
\end{barticle}
\endbibitem

\bibitem[\protect\citeauthoryear{Lauterbur}{1973}]{lauterbur1973image}
\begin{barticle}
\bauthor{\bsnm{Lauterbur}, \binits{P.C.}}:
\batitle{Image formation by induced local interactions: examples employing nuclear magnetic resonance}.
\bjtitle{nature}
\bvolume{242}(\bissue{5394}),
\bfpage{190}--\blpage{191}
(\byear{1973})
\end{barticle}
\endbibitem

\bibitem[\protect\citeauthoryear{Hedley and Rosenfeld}{1992}]{hedley1992new}
\begin{barticle}
\bauthor{\bsnm{Hedley}, \binits{M.}},
\bauthor{\bsnm{Rosenfeld}, \binits{D.}}:
\batitle{A new two-dimensional phase unwrapping algorithm for mri images}.
\bjtitle{Magnetic Resonance in Medicine}
\bvolume{24}(\bissue{1}),
\bfpage{177}--\blpage{181}
(\byear{1992})
\end{barticle}
\endbibitem

\bibitem[\protect\citeauthoryear{Pandit et~al.}{1994}]{pandit1994data}
\begin{barticle}
\bauthor{\bsnm{Pandit}, \binits{S.}},
\bauthor{\bsnm{Jordache}, \binits{N.}},
\bauthor{\bsnm{Joshi}, \binits{G.}}:
\batitle{Data-dependent systems methodology for noise-insensitive phase unwrapping in laser interferometric surface characterization}.
\bjtitle{JOSA A}
\bvolume{11}(\bissue{10}),
\bfpage{2584}--\blpage{2592}
(\byear{1994})
\end{barticle}
\endbibitem

\bibitem[\protect\citeauthoryear{Itoh}{1982}]{itoh1982analysis}
\begin{barticle}
\bauthor{\bsnm{Itoh}, \binits{K.}}:
\batitle{Analysis of the phase unwrapping algorithm}.
\bjtitle{Applied optics}
\bvolume{21}(\bissue{14}),
\bfpage{2470}--\blpage{2470}
(\byear{1982})
\end{barticle}
\endbibitem

\bibitem[\protect\citeauthoryear{Yu et~al.}{2019}]{yu2019phase}
\begin{barticle}
\bauthor{\bsnm{Yu}, \binits{H.}},
\bauthor{\bsnm{Lan}, \binits{Y.}},
\bauthor{\bsnm{Yuan}, \binits{Z.}},
\bauthor{\bsnm{Xu}, \binits{J.}},
\bauthor{\bsnm{Lee}, \binits{H.}}:
\batitle{Phase unwrapping in insar: A review}.
\bjtitle{IEEE Geoscience and Remote Sensing Magazine}
\bvolume{7}(\bissue{1}),
\bfpage{40}--\blpage{58}
(\byear{2019})
\end{barticle}
\endbibitem

\bibitem[\protect\citeauthoryear{Flynn}{1996}]{flynn1996consistent}
\begin{bchapter}
\bauthor{\bsnm{Flynn}, \binits{T.J.}}:
\bctitle{Consistent 2-d phase unwrapping guided by a quality map}.
In: \bbtitle{IGARSS'96. 1996 International Geoscience and Remote Sensing Symposium},
vol. \bseriesno{4},
pp. \bfpage{2057}--\blpage{2059}
(\byear{1996}).
\bcomment{IEEE}
\end{bchapter}
\endbibitem

\bibitem[\protect\citeauthoryear{Zhong et~al.}{2010}]{zhong2010improved}
\begin{barticle}
\bauthor{\bsnm{Zhong}, \binits{H.}},
\bauthor{\bsnm{Tang}, \binits{J.}},
\bauthor{\bsnm{Zhang}, \binits{S.}},
\bauthor{\bsnm{Chen}, \binits{M.}}:
\batitle{An improved quality-guided phase-unwrapping algorithm based on priority queue}.
\bjtitle{IEEE Geoscience and Remote Sensing Letters}
\bvolume{8}(\bissue{2}),
\bfpage{364}--\blpage{368}
(\byear{2010})
\end{barticle}
\endbibitem

\bibitem[\protect\citeauthoryear{Zhao et~al.}{2011}]{zhao2011quality}
\begin{barticle}
\bauthor{\bsnm{Zhao}, \binits{M.}},
\bauthor{\bsnm{Huang}, \binits{L.}},
\bauthor{\bsnm{Zhang}, \binits{Q.}},
\bauthor{\bsnm{Su}, \binits{X.}},
\bauthor{\bsnm{Asundi}, \binits{A.}},
\bauthor{\bsnm{Kemao}, \binits{Q.}}:
\batitle{Quality-guided phase unwrapping technique: comparison of quality maps and guiding strategies}.
\bjtitle{Applied optics}
\bvolume{50}(\bissue{33}),
\bfpage{6214}--\blpage{6224}
(\byear{2011})
\end{barticle}
\endbibitem

\bibitem[\protect\citeauthoryear{Jian}{2016}]{jian2016reliability}
\begin{barticle}
\bauthor{\bsnm{Jian}, \binits{G.}}:
\batitle{Reliability-map-guided phase unwrapping method}.
\bjtitle{IEEE Geoscience and Remote Sensing Letters}
\bvolume{13}(\bissue{5}),
\bfpage{716}--\blpage{720}
(\byear{2016})
\end{barticle}
\endbibitem

\bibitem[\protect\citeauthoryear{Goldstein et~al.}{1988}]{goldstein1988satellite}
\begin{barticle}
\bauthor{\bsnm{Goldstein}, \binits{R.M.}},
\bauthor{\bsnm{Zebker}, \binits{H.A.}},
\bauthor{\bsnm{Werner}, \binits{C.L.}}:
\batitle{Satellite radar interferometry: Two-dimensional phase unwrapping}.
\bjtitle{Radio science}
\bvolume{23}(\bissue{4}),
\bfpage{713}--\blpage{720}
(\byear{1988})
\end{barticle}
\endbibitem

\bibitem[\protect\citeauthoryear{AG}{2018}]{ag2018interferometric}
\begin{botherref}
\oauthor{\bsnm{AG}, \binits{G.R.S.}}:
Interferometric SAR processing: GAMMA remote sensing.
Gmligen, Switzerland
(2018)
\end{botherref}
\endbibitem

\bibitem[\protect\citeauthoryear{Ghiglia and Romero}{1994}]{ghiglia1994robust}
\begin{barticle}
\bauthor{\bsnm{Ghiglia}, \binits{D.C.}},
\bauthor{\bsnm{Romero}, \binits{L.A.}}:
\batitle{Robust two-dimensional weighted and unweighted phase unwrapping that uses fast transforms and iterative methods}.
\bjtitle{JOSA A}
\bvolume{11}(\bissue{1}),
\bfpage{107}--\blpage{117}
(\byear{1994})
\end{barticle}
\endbibitem

\bibitem[\protect\citeauthoryear{Ghiglia and Romero}{1996}]{ghiglia1996minimum}
\begin{barticle}
\bauthor{\bsnm{Ghiglia}, \binits{D.C.}},
\bauthor{\bsnm{Romero}, \binits{L.A.}}:
\batitle{Minimum lp-norm two-dimensional phase unwrapping}.
\bjtitle{JOSA A}
\bvolume{13}(\bissue{10}),
\bfpage{1999}--\blpage{2013}
(\byear{1996})
\end{barticle}
\endbibitem

\bibitem[\protect\citeauthoryear{Ferretti et~al.}{2007}]{ferretti2007insar}
\begin{bbook}
\bauthor{\bsnm{Ferretti}, \binits{A.}},
\bauthor{\bsnm{Monti-Guarnieri}, \binits{A.}},
\bauthor{\bsnm{Prati}, \binits{C.}},
\bauthor{\bsnm{Rocca}, \binits{F.}},
\bauthor{\bsnm{Massonet}, \binits{D.}}:
\bbtitle{InSAR Principles-guidelines for SAR Interferometry Processing and Interpretation}
vol. \bseriesno{19},
(\byear{2007})
\end{bbook}
\endbibitem

\bibitem[\protect\citeauthoryear{Yu et~al.}{2011}]{yu2011residues}
\begin{barticle}
\bauthor{\bsnm{Yu}, \binits{H.}},
\bauthor{\bsnm{Li}, \binits{Z.}},
\bauthor{\bsnm{Bao}, \binits{Z.}}:
\batitle{Residues cluster-based segmentation and outlier-detection method for large-scale phase unwrapping}.
\bjtitle{IEEE Transactions on Image Processing}
\bvolume{20}(\bissue{10}),
\bfpage{2865}--\blpage{2875}
(\byear{2011})
\end{barticle}
\endbibitem

\bibitem[\protect\citeauthoryear{Chen and Zebker}{2000}]{chen2000network}
\begin{barticle}
\bauthor{\bsnm{Chen}, \binits{C.W.}},
\bauthor{\bsnm{Zebker}, \binits{H.A.}}:
\batitle{Network approaches to two-dimensional phase unwrapping: intractability and two new algorithms}.
\bjtitle{JOSA A}
\bvolume{17}(\bissue{3}),
\bfpage{401}--\blpage{414}
(\byear{2000})
\end{barticle}
\endbibitem

\bibitem[\protect\citeauthoryear{Bioucas-Dias and Valadao}{2007}]{bioucas2007phase}
\begin{barticle}
\bauthor{\bsnm{Bioucas-Dias}, \binits{J.M.}},
\bauthor{\bsnm{Valadao}, \binits{G.}}:
\batitle{Phase unwrapping via graph cuts}.
\bjtitle{IEEE Transactions on Image processing}
\bvolume{16}(\bissue{3}),
\bfpage{698}--\blpage{709}
(\byear{2007})
\end{barticle}
\endbibitem

\bibitem[\protect\citeauthoryear{Flynn}{1997}]{flynn1997two}
\begin{barticle}
\bauthor{\bsnm{Flynn}, \binits{T.J.}}:
\batitle{Two-dimensional phase unwrapping with minimum weighted discontinuity}.
\bjtitle{JOSA A}
\bvolume{14}(\bissue{10}),
\bfpage{2692}--\blpage{2701}
(\byear{1997})
\end{barticle}
\endbibitem

\bibitem[\protect\citeauthoryear{Costantini}{1998}]{costantini1998novel}
\begin{barticle}
\bauthor{\bsnm{Costantini}, \binits{M.}}:
\batitle{A novel phase unwrapping method based on network programming}.
\bjtitle{IEEE Transactions on geoscience and remote sensing}
\bvolume{36}(\bissue{3}),
\bfpage{813}--\blpage{821}
(\byear{1998})
\end{barticle}
\endbibitem

\bibitem[\protect\citeauthoryear{Nico et~al.}{2000}]{nico2000bayesian}
\begin{barticle}
\bauthor{\bsnm{Nico}, \binits{G.}},
\bauthor{\bsnm{Palubinskas}, \binits{G.}},
\bauthor{\bsnm{Datcu}, \binits{M.}}:
\batitle{Bayesian approaches to phase unwrapping: theoretical study}.
\bjtitle{IEEE Transactions on Signal processing}
\bvolume{48}(\bissue{9}),
\bfpage{2545}--\blpage{2556}
(\byear{2000})
\end{barticle}
\endbibitem

\bibitem[\protect\citeauthoryear{Chen and Zebker}{2001}]{chen2001two}
\begin{barticle}
\bauthor{\bsnm{Chen}, \binits{C.W.}},
\bauthor{\bsnm{Zebker}, \binits{H.A.}}:
\batitle{Two-dimensional phase unwrapping with use of statistical models for cost functions in nonlinear optimization}.
\bjtitle{JOSA A}
\bvolume{18}(\bissue{2}),
\bfpage{338}--\blpage{351}
(\byear{2001})
\end{barticle}
\endbibitem

\bibitem[\protect\citeauthoryear{Dias and Leit{\~a}o}{2002}]{dias2002z}
\begin{barticle}
\bauthor{\bsnm{Dias}, \binits{J.M.}},
\bauthor{\bsnm{Leit{\~a}o}, \binits{J.M.}}:
\batitle{The z/spl pi/m algorithm: a method for interferometric image reconstruction in sar/sas}.
\bjtitle{IEEE Transactions on Image processing}
\bvolume{11}(\bissue{4}),
\bfpage{408}--\blpage{422}
(\byear{2002})
\end{barticle}
\endbibitem

\bibitem[\protect\citeauthoryear{}{}]{snap}
\begin{botherref}
{SNAP - ESA Sentinel Application Platform v2.0.2}.
\url{http://step.esa.int}
\end{botherref}
\endbibitem

\bibitem[\protect\citeauthoryear{Hooper et~al.}{2012}]{hooper2012recent}
\begin{barticle}
\bauthor{\bsnm{Hooper}, \binits{A.}},
\bauthor{\bsnm{Bekaert}, \binits{D.}},
\bauthor{\bsnm{Spaans}, \binits{K.}},
\bauthor{\bsnm{Ar{\i}kan}, \binits{M.}}:
\batitle{Recent advances in sar interferometry time series analysis for measuring crustal deformation}.
\bjtitle{Tectonophysics}
\bvolume{514},
\bfpage{1}--\blpage{13}
(\byear{2012})
\end{barticle}
\endbibitem

\bibitem[\protect\citeauthoryear{Black and Rangarajan}{1996}]{black1996unification}
\begin{barticle}
\bauthor{\bsnm{Black}, \binits{M.J.}},
\bauthor{\bsnm{Rangarajan}, \binits{A.}}:
\batitle{On the unification of line processes, outlier rejection, and robust statistics with applications in early vision}.
\bjtitle{International journal of computer vision}
\bvolume{19}(\bissue{1}),
\bfpage{57}--\blpage{91}
(\byear{1996})
\end{barticle}
\endbibitem

\bibitem[\protect\citeauthoryear{Daubechies et~al.}{2004}]{daubechies2004iterative}
\begin{barticle}
\bauthor{\bsnm{Daubechies}, \binits{I.}},
\bauthor{\bsnm{Defrise}, \binits{M.}},
\bauthor{\bsnm{De~Mol}, \binits{C.}}:
\batitle{An iterative thresholding algorithm for linear inverse problems with a sparsity constraint}.
\bjtitle{Communications on Pure and Applied Mathematics: A Journal Issued by the Courant Institute of Mathematical Sciences}
\bvolume{57}(\bissue{11}),
\bfpage{1413}--\blpage{1457}
(\byear{2004})
\end{barticle}
\endbibitem

\bibitem[\protect\citeauthoryear{Daubechies et~al.}{2010}]{daubechies2010iteratively}
\begin{barticle}
\bauthor{\bsnm{Daubechies}, \binits{I.}},
\bauthor{\bsnm{DeVore}, \binits{R.}},
\bauthor{\bsnm{Fornasier}, \binits{M.}},
\bauthor{\bsnm{G{\"u}nt{\"u}rk}, \binits{C.S.}}:
\batitle{Iteratively reweighted least squares minimization for sparse recovery}.
\bjtitle{Communications on Pure and Applied Mathematics: A Journal Issued by the Courant Institute of Mathematical Sciences}
\bvolume{63}(\bissue{1}),
\bfpage{1}--\blpage{38}
(\byear{2010})
\end{barticle}
\endbibitem

\bibitem[\protect\citeauthoryear{Bach et~al.}{2012}]{bach2012optimization}
\begin{barticle}
\bauthor{\bsnm{Bach}, \binits{F.}},
\bauthor{\bsnm{Jenatton}, \binits{R.}},
\bauthor{\bsnm{Mairal}, \binits{J.}},
\bauthor{\bsnm{Obozinski}, \binits{G.}}, \betal:
\batitle{Optimization with sparsity-inducing penalties}.
\bjtitle{Foundations and Trends{\textregistered} in Machine Learning}
\bvolume{4}(\bissue{1}),
\bfpage{1}--\blpage{106}
(\byear{2012})
\end{barticle}
\endbibitem

\bibitem[\protect\citeauthoryear{Mairal et~al.}{2014}]{mairal2014sparse}
\begin{barticle}
\bauthor{\bsnm{Mairal}, \binits{J.}},
\bauthor{\bsnm{Bach}, \binits{F.}},
\bauthor{\bsnm{Ponce}, \binits{J.}}, \betal:
\batitle{Sparse modeling for image and vision processing}.
\bjtitle{Foundations and Trends{\textregistered} in Computer Graphics and Vision}
\bvolume{8}(\bissue{2-3}),
\bfpage{85}--\blpage{283}
(\byear{2014})
\end{barticle}
\endbibitem

\bibitem[\protect\citeauthoryear{Fornasier et~al.}{2016}]{fornasier2016conjugate}
\begin{barticle}
\bauthor{\bsnm{Fornasier}, \binits{M.}},
\bauthor{\bsnm{Peter}, \binits{S.}},
\bauthor{\bsnm{Rauhut}, \binits{H.}},
\bauthor{\bsnm{Worm}, \binits{S.}}:
\batitle{Conjugate gradient acceleration of iteratively re-weighted least squares methods}.
\bjtitle{Computational optimization and applications}
\bvolume{65},
\bfpage{205}--\blpage{259}
(\byear{2016})
\end{barticle}
\endbibitem

\bibitem[\protect\citeauthoryear{Nocedal and Wright}{2006}]{nocedal2006conjugate}
\begin{botherref}
\oauthor{\bsnm{Nocedal}, \binits{J.}},
\oauthor{\bsnm{Wright}, \binits{S.J.}}:
Conjugate gradient methods.
Numerical optimization,
101--134
(2006)
\end{botherref}
\endbibitem

\bibitem[\protect\citeauthoryear{Ahuja et~al.}{1988}]{ahuja1988network}
\begin{botherref}
\oauthor{\bsnm{Ahuja}, \binits{R.K.}},
\oauthor{\bsnm{Magnanti}, \binits{T.L.}},
\oauthor{\bsnm{Orlin}, \binits{J.B.}}:
Network flows
(1988)
\end{botherref}
\endbibitem

\bibitem[\protect\citeauthoryear{Bertsekas}{2014}]{bertsekas2014constrained}
\begin{bbook}
\bauthor{\bsnm{Bertsekas}, \binits{D.P.}}:
\bbtitle{Constrained Optimization and Lagrange Multiplier Methods}.
\bpublisher{Academic press},
\blocation{Cambridge, MA}
(\byear{2014})
\end{bbook}
\endbibitem

\bibitem[\protect\citeauthoryear{Beck and Teboulle}{2009}]{beck2009fast}
\begin{barticle}
\bauthor{\bsnm{Beck}, \binits{A.}},
\bauthor{\bsnm{Teboulle}, \binits{M.}}:
\batitle{A fast iterative shrinkage-thresholding algorithm for linear inverse problems}.
\bjtitle{SIAM journal on imaging sciences}
\bvolume{2}(\bissue{1}),
\bfpage{183}--\blpage{202}
(\byear{2009})
\end{barticle}
\endbibitem

\bibitem[\protect\citeauthoryear{Blumensath and Davies}{2009}]{blumensath2009iterative}
\begin{barticle}
\bauthor{\bsnm{Blumensath}, \binits{T.}},
\bauthor{\bsnm{Davies}, \binits{M.E.}}:
\batitle{Iterative hard thresholding for compressed sensing}.
\bjtitle{Applied and computational harmonic analysis}
\bvolume{27}(\bissue{3}),
\bfpage{265}--\blpage{274}
(\byear{2009})
\end{barticle}
\endbibitem

\bibitem[\protect\citeauthoryear{Lawson}{1961}]{lawson1961contribution}
\begin{botherref}
\oauthor{\bsnm{Lawson}, \binits{C.L.}}:
Contribution to the theory of linear least maximum approximation.
Ph. D. dissertation. Univ. Calif.
(1961)
\end{botherref}
\endbibitem

\bibitem[\protect\citeauthoryear{Holland and Welsch}{1977}]{holland1977robust}
\begin{barticle}
\bauthor{\bsnm{Holland}, \binits{P.W.}},
\bauthor{\bsnm{Welsch}, \binits{R.E.}}:
\batitle{Robust regression using iteratively reweighted least-squares}.
\bjtitle{Communications in Statistics-theory and Methods}
\bvolume{6}(\bissue{9}),
\bfpage{813}--\blpage{827}
(\byear{1977})
\end{barticle}
\endbibitem

\bibitem[\protect\citeauthoryear{Osborne}{1985}]{osborne1985finite}
\begin{bbook}
\bauthor{\bsnm{Osborne}, \binits{M.R.}}:
\bbtitle{Finite Algorithms in Optimization and Data Analysis}.
\bpublisher{John Wiley \& Sons, Inc.},
\blocation{Hoboken, NJ}
(\byear{1985})
\end{bbook}
\endbibitem

\bibitem[\protect\citeauthoryear{Beck}{2015}]{beck2015convergence}
\begin{barticle}
\bauthor{\bsnm{Beck}, \binits{A.}}:
\batitle{On the convergence of alternating minimization for convex programming with applications to iteratively reweighted least squares and decomposition schemes}.
\bjtitle{SIAM Journal on Optimization}
\bvolume{25}(\bissue{1}),
\bfpage{185}--\blpage{209}
(\byear{2015})
\end{barticle}
\endbibitem

\bibitem[\protect\citeauthoryear{GERSCHGORIN}{1931}]{gerschgorin1931uber}
\begin{barticle}
\bauthor{\bsnm{GERSCHGORIN}, \binits{S.}}:
\batitle{Uber die abgrenzung der eigenwerte einer matrix}.
\bjtitle{lzv. Akad. Nauk. USSR. Otd. Fiz-Mat. Nauk}
\bvolume{7},
\bfpage{749}--\blpage{754}
(\byear{1931})
\end{barticle}
\endbibitem

\bibitem[\protect\citeauthoryear{Kaasschieter}{1988}]{kaasschieter1988preconditioned}
\begin{barticle}
\bauthor{\bsnm{Kaasschieter}, \binits{E.F.}}:
\batitle{Preconditioned conjugate gradients for solving singular systems}.
\bjtitle{Journal of Computational and Applied mathematics}
\bvolume{24}(\bissue{1-2}),
\bfpage{265}--\blpage{275}
(\byear{1988})
\end{barticle}
\endbibitem

\bibitem[\protect\citeauthoryear{Hayami}{2018}]{hayami2018convergence}
\begin{botherref}
\oauthor{\bsnm{Hayami}, \binits{K.}}:
Convergence of the conjugate gradient method on singular systems.
arXiv preprint arXiv:1809.00793
(2018)
\end{botherref}
\endbibitem

\bibitem[\protect\citeauthoryear{Kelley}{1995}]{kelley1995iterative}
\begin{bbook}
\bauthor{\bsnm{Kelley}, \binits{C.T.}}:
\bbtitle{Iterative Methods for Linear and Nonlinear Equations}.
\bpublisher{SIAM},
\blocation{Philadelphia, PA}
(\byear{1995})
\end{bbook}
\endbibitem

\bibitem[\protect\citeauthoryear{Bartels and Stewart}{1972}]{bartels1972solution}
\begin{barticle}
\bauthor{\bsnm{Bartels}, \binits{R.H.}},
\bauthor{\bsnm{Stewart}, \binits{G.W.}}:
\batitle{Solution of the matrix equation ax+ xb= c [f4]}.
\bjtitle{Communications of the ACM}
\bvolume{15}(\bissue{9}),
\bfpage{820}--\blpage{826}
(\byear{1972})
\end{barticle}
\endbibitem

\bibitem[\protect\citeauthoryear{Chen and Zebker}{2002}]{chen2002phase}
\begin{barticle}
\bauthor{\bsnm{Chen}, \binits{C.W.}},
\bauthor{\bsnm{Zebker}, \binits{H.A.}}:
\batitle{Phase unwrapping for large sar interferograms: Statistical segmentation and generalized network models}.
\bjtitle{IEEE Transactions on Geoscience and Remote Sensing}
\bvolume{40}(\bissue{8}),
\bfpage{1709}--\blpage{1719}
(\byear{2002})
\end{barticle}
\endbibitem

\bibitem[\protect\citeauthoryear{Hanssen}{2002}]{Hanssen2002}
\begin{bbook}
\bauthor{\bsnm{Hanssen}, \binits{R.}}:
\bbtitle{Radar Interferometry, Data Interpretation and Error Analysis}.
\bpublisher{Kluwer Academic Publishers},
\blocation{Dordrecht, Netherlands}
(\byear{2002})
\end{bbook}
\endbibitem

\bibitem[\protect\citeauthoryear{Farr et~al.}{2007}]{farr2007shuttle}
\begin{botherref}
\oauthor{\bsnm{Farr}, \binits{T.G.}},
\oauthor{\bsnm{Rosen}, \binits{P.A.}},
\oauthor{\bsnm{Caro}, \binits{E.}},
\oauthor{\bsnm{Crippen}, \binits{R.}},
\oauthor{\bsnm{Duren}, \binits{R.}},
\oauthor{\bsnm{Hensley}, \binits{S.}},
\oauthor{\bsnm{Kobrick}, \binits{M.}},
\oauthor{\bsnm{Paller}, \binits{M.}},
\oauthor{\bsnm{Rodriguez}, \binits{E.}},
\oauthor{\bsnm{Roth}, \binits{L.}}, et al.:
The shuttle radar topography mission.
Reviews of geophysics
\textbf{45}(2)
(2007)
\end{botherref}
\endbibitem

\bibitem[\protect\citeauthoryear{Linde-Cerezo et~al.}{2021}]{Linde-Cerezo2021}
\begin{bchapter}
\bauthor{\bsnm{Linde-Cerezo}, \binits{A.}},
\bauthor{\bsnm{Rodriguez-Cassola}, \binits{M.}},
\bauthor{\bsnm{Prats-Iraola}, \binits{P.}},
\bauthor{\bsnm{Pinheiro}, \binits{M.}}:
\bctitle{{Systematic Comparison of Backgeocoding Algorithms for SAR Processing and Simulation Environments}}.
In: \bbtitle{EUSAR 2021; 13th European Conference on Synthetic Aperture Radar},
pp. \bfpage{383}--\blpage{386}
(\byear{2021})
\end{bchapter}
\endbibitem

\bibitem[\protect\citeauthoryear{Akiki et~al.}{2022}]{akiki2022improved}
\begin{bchapter}
\bauthor{\bsnm{Akiki}, \binits{R.}},
\bauthor{\bsnm{Anger}, \binits{J.}},
\bauthor{\bsnm{{De Franchis}}, \binits{C.}},
\bauthor{\bsnm{Facciolo}, \binits{G.}},
\bauthor{\bsnm{Morel}, \binits{J.M.}},
\bauthor{\bsnm{Grandin}, \binits{R.}}:
\bctitle{{Improved Sentinel-1 IW Burst Stitching Through Geolocation Error Correction Considerations}}.
In: \bbtitle{International Geoscience and Remote Sensing Symposium (IGARSS)},
pp. \bfpage{3404}--\blpage{3407}.
\bpublisher{IEEE},
\blocation{New York, NY}
(\byear{2022})
\end{bchapter}
\endbibitem

\bibitem[\protect\citeauthoryear{Braun}{2021}]{Braun2021}
\begin{barticle}
\bauthor{\bsnm{Braun}, \binits{A.}}:
\batitle{{Retrieval of digital elevation models from Sentinel-1 radar data – open applications, techniques, and limitations}}.
\bjtitle{Open Geosciences}
\bvolume{13}(\bissue{1}),
\bfpage{532}--\blpage{569}
(\byear{2021})
\end{barticle}
\endbibitem

\bibitem[\protect\citeauthoryear{Goldstein and Werner}{1998}]{Goldstein1998}
\begin{barticle}
\bauthor{\bsnm{Goldstein}, \binits{R.M.}},
\bauthor{\bsnm{Werner}, \binits{C.L.}}:
\batitle{{Radar interferogram filtering for geophysical applications}}.
\bjtitle{Geophysical Research Letters}
\bvolume{25}(\bissue{21}),
\bfpage{4035}--\blpage{4038}
(\byear{1998})
\end{barticle}
\endbibitem

\bibitem[\protect\citeauthoryear{Chen}{2001}]{chen2001statistical}
\begin{bbook}
\bauthor{\bsnm{Chen}, \binits{C.W.}}:
\bbtitle{Statistical-cost Network-flow Approaches to Two-dimensional Phase Unwrapping for Radar Interferometry}.
\bpublisher{stanford university},
\blocation{Stanford, CA}
(\byear{2001})
\end{bbook}
\endbibitem

\end{thebibliography}

\end{document}